\documentclass[leqno,final]{siamltex}
\usepackage{graphicx}
\usepackage{subfigure}
\usepackage{epstopdf}
\usepackage{amsmath,amstext,amssymb,hyperref}
\usepackage{color,xcolor}
\usepackage{algorithm}
\usepackage{multicol}
\allowdisplaybreaks[2]

\numberwithin{equation}{section}

\newcommand{\norm}[1]{\left\Vert#1\right\Vert}

\newcommand{\norme}[1]{\left\Vert{\hskip -2.7pt}\left\vert #1 \right\vert{\hskip -2.7pt}\right\Vert}

\newcommand{\abs}[1]{\left\vert#1\right\vert}
\newcommand{\pd}[1]{\left\langle #1\right\rangle}

\newcommand{\set}[1]{\left\{#1\right\}}

\newcommand{\til}[1]{\tilde{#1}}

\newcommand{\R}{\mathbb{R}}

\newcommand{\al}{\alpha}
\newcommand{\be}{\beta}

\newcommand{\De}{\Delta}

\newcommand{\ga}{\gamma}
\newcommand{\Ga}{\Gamma}

\newcommand{\na}{\nabla}

\newcommand{\Om}{\Omega}
\newcommand{\pa}{\partial}

\newcommand{\si}{\sigma}

\newcommand{\ze}{\zeta}

\newcommand{\dt}{\,\mathrm{d} t}

\renewcommand{\i}{{\rm\mathbf i}}

\newcommand{\pml}{\mathrm{PML}}
\newcommand{\tru}{\mathrm{tru}}

\DeclareMathOperator{\im}{{Im}}

\newcommand{\eq}[1]{\begin{align}#1\end{align}}
\newcommand{\eqn}[1]{\begin{align*}#1\end{align*}}

\newcommand{\diam}{\mathrm{diam}}

\title{A pure source transfer domain decomposition method for Helmholtz equations in unbounded domain}
\author{
 Yu Du\footnotemark[1]
 \thanks{Department of Mathematics, Xiangtan University, Hunan, 411105, China. ({\tt duyu@xtu.edu.cn}). The research of this author was supported in part by the Natural Science Foundation of China under grants 11601026 and the Hunan Provincial Natural Science Foundation of China (NO. 2019JJ50572)}
 \and
 Haijun Wu\footnotemark[2]
 \thanks{Department of Mathematics, Nanjing University, Jiangsu, 210093, P.R. China. ({\tt hjw@nju.edu.cn}). This research was partially supported by the Natural Science Foundation of China under grants 11525103 and 91130004.}
}

\begin{document}

\maketitle

\vspace{-1.4in}
\slugger{sinum}{200x}{xx}{x}{xxx--xxx}
\vspace{1.4in}

\setcounter{page}{1}

\begin{abstract} We propose a pure source transfer domain decomposition method (PSTDDM) for solving the truncated perfectly matched layer (PML) approximation in bounded domain of Helmholtz scattering problem. The method is a modification of the STDDM proposed by [Z. Chen and X. Xiang, SIAM J. Numer. Anal., 51 (2013), pp. 2331--2356]. After decomposing the domain into $N$ non-overlapping layers, the STDDM is composed of two series steps of sources transfers and wave expansions, where $N-1$ truncated PML problems on two adjacent layers and $N-2$ truncated half-space PML problems are solved successively. While the PSTDDM consists merely of two parallel source transfer steps in two opposite directions, and in each step $N-1$ truncated PML problems on two adjacent layers are solved successively. One benefit of such a modification is that the truncated PML problems on two adjacent layers can be further solved by the PSTDDM along directions parallel to the layers. And therefore, we obtain a block-wise PSTDDM on the decomposition composed of $N^2$ squares, which reduces the size of subdomain problems and is more suitable for large-scale problems. Convergences of both the layer-wise PSTDDM and the block-wise PSTDDM are proved for the case of constant wave number. Numerical examples are included to show that the PSTDDM gives good approximations to the discrete Helmholtz equations with constant wave numbers and can be used as an efficient preconditioner in the preconditioned GMRES method for solving the discrete Helmholtz equations with constant and heterogeneous wave numbers.
\end{abstract}



\setcounter{page}{1}

\section{Introduction}\label{sec_int} 	This paper is devoted to a domain decomposition method for the Helmholtz problem in the full space $\R^2$ with Sommerfeld radiation condition:
\begin{align}
	\De u + k^2 u =f \qquad     & \mbox{in } \R^2,\label{eq1.1a}   \\
	\abs{\frac{\pa u}{\pa r}-\i ku} =o(r^{-\frac12}) \qquad & \mbox{as } r=\abs{x}\rightarrow\infty.\label{eq1.1b}
\end{align}
where the wave number $k$ is positive and $f\in L^2(\R^2)$ having compact support. The proposed domain decomposition method is a modification of the source transfer domain decomposition method (STDDM) given in \cite{cx}, which can reduce further the size of subdomain problems, can be easily extended to solve three-dimensional Helmholtz scattering problems, and can be used as an efficient preconditioner for discrete Helmholtz problems with constant and heterogeneous wave numbers.

The Helmholtz problem \eqref{eq1.1a}-\eqref{eq1.1b} appears in various applications, such as acoustic, elastic, and electromagnetic scattering problems. Diverse numerical methods have been proposed or analyzed for Helmholtz equations with large wave numbers (see, e.g., \cite{Ainsworth04, bips95,bs00,dbb99,harari97, ib95a,ib97,melenk95a,ms10,ms11,dw,lw19,w,zw,dgmz12,fw09, fw11, mps13,cd03, ghp09,clx13,gm,M2014Localization,Peterseim2014,ov18, thompson06}). Due to the highly indefinite nature of the Helmholtz problem with large wave number, it is challenging to solve the linear algebraic systems resulting from its numerical discretization, which usually contain huge numbers of degrees of freedom, in particular for the cases of three dimensions or/and existing pollution effects \cite{bips95,bs00, dbb99, ib95a,ib97, thompson06}. Considerable efforts in the literature have been made on fast solvers, e.g., multigrid methods \cite{BL,cwx15,EEO} and domain decomposition methods \cite{BD,clx16,GMN,gz13,lj19,lx19}. Recently Engquist and Ying \cite{ey11b} introduced a new sweeping preconditioner for the iterative solution of the Helmholtz equation, which has linear application cost, and the preconditioned iterative solver converges in a number of iterations that is essentially independent of the number of unknowns or the frequency (see also \cite{ly16,eg16,Stolk2017,gz19}). Inspired by \cite{ey11b}, Chen and Xiang \cite{cx} proposed a source transfer domain decomposition method (STDDM) along with a rigorous convergence analysis, which can be used as an efficient preconditioner in the preconditioned GMRES method for solving discrete Helmholtz equations. After decomposing the domain into $N$ non-overlapping layers, the STDDM is composed of two series steps of sources transfers and wave expansions, where $N-1$ truncated PML problems on two adjacent layers and $N-2$ truncated half-space PML problems are solved successively. We also refer to \cite{Xiang2019} for a double STDDM.

In this paper we modify the STDDM in \cite{cx} to obtain a domain decomposition method that consists merely of two parallel source transfer steps in two opposite directions, and in each step $N-1$ truncated PML problems on two adjacent layers are solved successively.
Therefore, we call it ``pure source transfer domain decomposition method" (PSTDDM). Next we illustrate the ideas of PSTDDM by considering the Helmholtz problem \eqref{eq1.1a}--\eqref{eq1.1b}.
Let $B_l=\{x=(x_1,x_2)\in\R^2:\abs{x_1}<l_1,\abs{x_2}<l_2\}$. Assume that $f$ is supported in $B_l$. We divide the interval $(-l_1,l_1)$ into $N$ segments with the points $\ze_i=-l_1+(i-1)\De\ze$ where $\De\ze=2l_1/N$ and $1\le i\le N+1$. Set $\ze_0=-\infty, \ze_{N+2}=\infty$. Then we denote the layers by
\begin{align}
	\Omega_1 = & \set{ x=(x_1,x_2)\in\R^2: x_1<\ze_{2} },\notag     \\
	\Omega_i = & \{ x=(x_1,x_2)\in\R^2:\ze_i<x_1<\ze_{i+1} \},\ i=2,\cdots,N-1,\label{eq_Nlayers} \\
	\Omega_{N} = & \set{ x=(x_1,x_2)\in\R^2: x_1>\ze_{N} }.\notag
\end{align}
Clearly, $\supp f\subset\cup_{i=1}^N\Omega_i$. Let $f_i=f$ in $\Omega_i$ and $f_i=0$ elsewhere. Let $\bar{f}_1^+=f_1$ and $\bar{f}_N^-=f_N$. Similar to \cite{cx}, the key idea is that by defining the source transfer operators $\Psi_i^\pm$ in the sense that
\begin{align*}
	 & \int_{\Omega_i} \bar{f}_i^+(y)G(x,y)dy = \int_{\Omega_{i+1}} \Psi_{i+1}^+(\bar{f}_i^+)(y)G(x,y)dy\quad \forall x\in\cup_{j=i+2}^{N+1}\Omega_j, \\
	 & \int_{\Omega_i} \bar{f}_i^-(y)G(x,y)dy = \int_{\Omega_{i-1}} \Psi_{i-1}^-(\bar{f}_i^-)(y)G(x,y)dy\quad \forall x\in\cup_{j=0}^{i-2}\Omega_j,
\end{align*}
and letting $\bar{f}_{i\pm1}^\pm=f_{i\pm1}+\Psi_{i\pm1}^\pm(\bar{f}_i)$. Clearly, we have for any $x\in\Omega_i$
\begin{align}\label{eqi1}
	u(x) & = \bigg(-\int_{\Omega_i}f_i(y)G(x,y)dy-\int_{\Omega_{i-1}}\bar{f}_{i-1}^+(y)G(x,y)dy\bigg) + \\
	 & \bigg(-\int_{\Omega_{i+1}}\bar{f}_{i+1}^-(y)G(x,y)dy\bigg).\notag
\end{align}
Here, $G(x,y)$ is the Green's function of the problem \eqref{eq1.1a}--\eqref{eq1.1b}. Observing \eqref{eqi1}, we know that $u(x)$ in $\Omega_i$ consists of two independent parts. The first part only involves the sources in $\Omega_i$ and $\Omega_{i-1}$ and the second one only involves the source in $\Omega_{i+1}$. Thus they could be solved independently by using the perfectly matched layer (PML) method right outside $\Omega_{i-1}\cup\Omega_i$ and $\Omega_i\cup\Omega_{i+1}$, respectively. The above procedure leads to the PSTDDM of Algorithm~\ref{alg1} in Subsection~\ref{subs1}. The PSTDDM with truncated PML is listed in Algorithm~\ref{alg2} and its convergence is proved in Subsection~\ref{sec3}. Moreover, the PML problems on two adjacent layers in the above PSTDDM can be further solved by the same PSTDDM but along directions parallel to the layers. And therefore, we obtain some block-wise PSTDDM on the decomposition composed of $N^2$ squares (see Algorithms~\ref{alg3} and \ref{alg4} in Section~\ref{sec4}), which further reduces the size of subdomain problems and is more suitable for large-scale problems, in particular in 3D. For the sake of clarity, we sometimes call the PSTDDM in Algorithms~\ref{alg1} and \ref{alg2} the layer-wise PSTDDM.

The PML is a mesh termination technique of effectiveness, simplicity and flexibility in computational wave propagation. After the pioneering work of B\'{e}renger \cite{ber94,ber96}, various constructions of PML absorbing layers have been proposed and many theoretical results about Helmholtz problem, such as those about the convergence and stability, have been studied \cite{bp,cl03,cw03,cw08,cz11,kp10,ls98,ls01}. In this paper, the uniaxial PML methods will be used (see e.g. \cite{bp,cx}).

The remainder of this paper is organized as follows. In Section \ref{sec2}, we first recall the uniaxial PML and related properties. Then we introduce the layer-wise PSTDDM with non-truncated PML and show that it just produces the exact solution by solving PML problems on two adjacent layers. Thirdly, we give the layer-wise PSTDDM with truncated PML and prove that it converges exponentially with respect to the medium parameters or the thickness of the PML. Section \ref{sec4} is devoted to the block-wise PSTDDM. In particular, an error estimate with explicit dependences on $k$ and $N$ for the block-wise PSTDDM with truncated PML is derived. In Section~\ref{sec5}, the layer-wise and block-wise PSTDDM with truncated PML are discretized by bilinear finite element methods. Numerical experiments are presented to show the accuracies of the discretized PSTDDM and the performances of using the discretized PSTDDM as preconditioners in the preconditioned GMRES when the wave number or the number of subdomains (layers or blocks) increases.

Throughout the paper, for any bounded domain $U$, we shall use the standard Sobolev space $H^s(U)$, its norm and inner product,
and refer to \cite{bs08} for their definitions.
But, we will often use the following weighted norms
\begin{align}
	\norme{v}_U   & = \left(\norm{\na v}_{L^2(U)}^2+\norm{k v}_{L^2(U)}^2\right)^{\frac12},        \\
	\norm{v}_{H^{\frac12}(\pa U)} & = \left(d_U^{-1}\norm{v}_{L^2(\pa U)}^2+\abs{v}_{{H^{\frac12}(\pa U)}}^2\right)^{\frac12},\quad\forall v\in H^1(U), \label{eq_vhga12}
\end{align}
where $d_U:=\diam(U)$. For any $f\in H^1(U)'$ the dual space of $H^1(U)$, define its norm
\begin{align*}
	\norme{f}_U^*:=\sup_{0\neq v\in H^1(U)}\frac{\pd{f,v}}{\norme{v}_U}.
\end{align*}
It is clear that $\norme{f}_U^*\leq k^{-1} \norm{f}_{L^2(U)}$ if $f\in L^2(U)$.

For the simplicity of notation, we shall frequently use
$C$ for a generic positive constant in most of the subsequent estimates,
which is independent of $k$. We will also often write $A\lesssim B$ and $B\gtrsim A$ for the inequalities $A\leq C B$ and $B\geq CA$ respectively. $A\eqsim B$ is used for an equivalent statement when both $A\lesssim B$ and $B\lesssim A$ hold. We suppose that $k\gtrsim 1$.

\section{Source transfer layer by layer} \label{sec2}
In this section, we first recall the PML formulations for the Helmholtz scattering problem \eqref{eq1.1a}--\eqref{eq1.1b}. Then we introduce the PSTDDM for the PML problem in the whole space as well as in the bounded truncated domain by dividing the computational domain into layers and doing source transfer layer by layer.
\subsection{The PML problems}\label{ssecpml}
In this subsection, we introduce the PML to truncate the unbounded domain. We adopt the uniaxial PML method \cite{bp,cw,cx,kp10}. The model medium properties are defined by
\begin{align*}
	 & \al_1(x_1)=1+\i\si_1(x_1),\ \al_2(x_2)=1+\i\si_2(x_2)     \\
	 & \si_j(t)=\si_j(-t) \text{ and } 0\le \si_j(t)\le\si_0\ \mathrm{for}\ t\in\R^2,  \\
	 & \si_j=0\ \mathrm{for}\ \abs{t}\leq l_j,\ \si_j=\si_0>0\ \mathrm{for}\ \abs{t}\geq \bar{l}_j.
\end{align*}
where $\si_j(x_j)\in C^1(\R^2)$ are piecewise smooth functions, $\bar{l}_j>l_j$ is fixed and $\si_0$ is a constant. We remark that the requirements on $\si_j$ are used in \cite{cx} to estimate the inf--sup constants of the PML problems.
We introduce the PML by complex coordinate stretching \cite{cjm97,cx} as follows.
For $x=(x_1,x_2)^T$, define $\tilde{x}(x)=(\tilde{x}_1(x_1),\tilde{x}_2(x_2))$ by
\begin{align}
	\tilde{x}_j(x_j) = \int_0^{x_j} \al_j(t) dt = x_j + \i\int_0^{x_j}\si_j(t)dt,\ j=1,2. \label{eq_defPML}
\end{align}
Let $\tilde\De$ denote the Laplacian with respect to $\tilde x$. Then the PML equation is derived from $\tilde{\De}\tilde{u}+k^2\tilde{u}=f$ in $\R^2$ by using the chain rule:
\begin{align}\label{pmleq}
	J^{-1}\na\cdot(A\na\tilde{u})+k^2\tilde{u}=f\quad \mathrm{in}\ \R^2,
\end{align}
with the radiation condition that $\tilde u$ is bounded in $\R^2$. Here $A(x)=\diag\left(\frac{\al_2(x_2)}{\al_1(x_1)},\frac{\al_1(x_1)}{\al_2(x_2)}\right)$ and $J(x)=\al_1(x_1)\al_2(x_2)$. The weak formulation of \eqref{pmleq} reads as: Find $\tilde u\in H^1(\R^2)$ such that
\begin{align}
	(A\na \tilde{u},\na v) - k^2 (J\tilde{u},v) = - \pd{Jf,v}\ \forall v\in H^1(\R^2), \label{eq_exactsol}
\end{align}
where $(\cdot,\cdot)$ is the inner product in $L^2(\R^2)$ and $\pd{\cdot,\cdot}$ is the duality pairing between $H^1(\R^2)'$ and $H^1(\R^2)$.

We recall that the exact solution of the Helmholtz problem \eqref{eq1.1a}--\eqref{eq1.1b} can be written as the acoustic volume potential. Let $G(x,y):=\frac{\i}{4}H_0^{(1)}(k\abs{x-y})$ be the fundamental solution of the Helmholtz problem
where $H_0^{(1)}$ is the first kind Hankel function of order zero.
Then, the solution of \eqref{eq1.1a} is given by
\begin{align}
	u(x)=-\int_{\R^2}f(y) G(x,y)dy\quad \forall x\in\R^2.
\end{align}
Similarly,	the solution to the PML problem can be expressed as
\begin{align}\label{tu}
	\tilde{u}(x)=u(\tilde{x})=- \int_{\R^2} f(y)G(\tilde{x},\tilde{y}) dy\quad \forall x\in\R^2.
\end{align}
Note that $\tilde{y}=y$ and $\tilde{u}=u$ in $B_l$ since $f$ is supported inside $B_l$. For any $z\in\mathbb{C}\setminus[0,+\infty)$, denote by $z^\frac12$ the analytic branch of $\sqrt{z}$ such that $\mathrm{Re}(z^\frac12)>0$. Define the ``complex distance" \cite{cx}
\eq{\label{rho}
	\rho(\tilde{x},\tilde{y})=\big((\tilde x_1-\tilde y_1)^2+(\tilde x_2-\tilde y_2)^2\big)^\frac12
}
The fundamental solution of the PML equation \eqref{pmleq} is formulated as (see \cite{cx,ls01}):
\begin{align}\label{fs2}
	\tilde{G}(x,y)=J(y)G(\tilde{x},\tilde{y})=\frac{\i}{4}J(y)H_0^{(1)}(k\rho(\tilde{x},\tilde{y})).
\end{align}

It is proved that $\tilde u$ decays exponentially as $|x|\to\infty$ \cite{cl03,bp,cx,ls01}.	Therefore the PML can be truncated where $\tilde u$ small enough. Let $B_L=(-l_1-d_1,l_1+d_1)\times(-l_2-d_2,l_2+d_2)$ where $l_1+d_1>\bar{l}_1$ and $l_2+d_2>\bar{l}_2$. Then we have the following truncated PML problem obtained by truncate the PML at $\pa B_L$:
Find $\hat u\in H_0^1(B_L)$ such that
\begin{align}
	(A\na \hat{u},\na v) - k^2 (J\hat{u},v) = - \pd{Jf,v}\ \forall v\in H_0^1(B_L). \label{eq_tpml}
\end{align}
Denote by $d=\min(d_1,d_2)$. The following lemma says that the the sesquilinear form associated with the truncated PML problem satisfies the inf--sup condition and that the truncated PML solution is exponentially close to the PML solution. We refer to \cite[Lemmas 3.3, 3.4, and 3.6]{cx} for the proof.

\begin{lemma}\label{Linfsup}
	For sufficiently large $\si_0 d\ge 1$, there exists a constant $\al\ge 0$ such that the truncated PML problem \eqref{eq_tpml} attains a unique solution and the following inf-sup condition holds.
	\begin{align}\label{tinfsup}
		\sup_{\psi\in H_0^1(B_L)} \frac{\abs{(A\na\phi,\na\psi)-k^2(J\phi,\psi)}}{\norme{\psi}_{B_L}}\geq \mu \norme{\phi}_{B_L}\quad \forall \phi\in H_0^1(B_L),
	\end{align}
	where $\mu^{-1}\leq C(\sigma_0) k^{1+\al}$ and $C(\sigma_0)$ is a constant that may polynomially depend on $\sigma_0$.
\end{lemma}

{\it Remark 2.1.} (i)
\cite[Lemma 3.4]{cx} shows that the above lemma holds with $\al=\frac12$. But since the inf-sup constant for the original scattering problem is of order $O(k^{-1})$ \cite{cm08}, we expect the lemma holds with $\al=0$. Actually this has been proved for the circular PML \cite{lw19} although it has not been proved for the uniaxial PML setting of this paper yet.

(ii) The inf--sup condition holds also for the PML equation in $\R^2$ \cite[(3.10)]{cx}:
\begin{align}\label{tinfsup2}
	\sup_{\psi\in H^1(\R^2)} \frac{(A\na\phi,\na\psi)-k^2(J\phi,\psi)}{\norme{\psi}_{\R^2}}\gtrsim \mu \norme{\phi}_{\R^2} \quad \forall \phi\in H^1(\R^2).
\end{align}

Finally, we recall estimates of the Green function, see \cite[(2.5) and Lemma~2.5]{cx}.
\begin{lemma}\label{lemma_G}
	For any $x,y\in\R^2, x\neq y$, there hold
	\begin{align*}
		\abs{G(\tilde{x},\tilde{y})}  & \lesssim e^{-\frac12 k\im\rho(\tilde{x},\tilde{y})}(k\abs{x-y})^{-\frac12},    \\
		\abs{\na_xG(\tilde{x},\tilde{y})} & \lesssim ke^{-\frac12 k\im\rho(\tilde{x},\tilde{y})}\big((k\abs{x-y})^{-\frac12}+(k\abs{x-y})^{-1}\big), \\
		\abs{\na_x\na_yG(\tilde{x},\tilde{y})} & \lesssim k^2e^{-\frac12 k\im\rho(\tilde{x},\tilde{y})}\big((k\abs{x-y})^{-\frac12}+(k\abs{x-y})^{-2}\big).
	\end{align*}
	Moreover,
	\begin{align}
		\im\rho(\tilde{x},\tilde{y})\ge\frac{1}{\abs{x-y}}\bigg(\abs{x_1-y_1}\bigg|\int_{y_1}^{x_1}\si_1(t) dt\bigg|+\abs{x_2-y_2}\bigg|\int_{y_2}^{x_2}\si_2(t)dt\bigg|\bigg). \label{eq_irxyi}
	\end{align}
\end{lemma}

\subsection{The PSTDDM for the PML problem in $\R^2$}\label{subs1}
In this subsection, we introduce PSTDDM for the PML equation in the whole space and give the fundamental theorems.

We recall that the domain $\set{x=(x_1,x_2):\abs{x_1}\leq l_1}$ is divided into $N$ layers \eqref{eq_Nlayers} and
that $f_i(x)=f(x)|_{\Omega_i}$ for any $x\in\Omega_i$ and $f_i(x)=0$ for any $x\in\R^2\backslash\bar{\Omega}_i$. We define smooth functions ${\be_i^+}(x_1)$ and ${\be_i^-}(x_1)$ by
\begin{align}
	 & {\be_i^+}=1,\ {\be_i^-}=0,\ {\be_i^+}'={\be_i^-}'=0\ \mathrm{as}\ x_1\leq\ze_i, \label{beta1} \\
	 & {\be_i^+}=0,\ {\be_i^-}=1,\ {\be_i^+}'={\be_i^-}'=0\ \mathrm{as}\ x_1\geq\ze_{i+1},\notag \\
	 & \abs{{\be_i^+}'}\leq C(\De\ze)^{-1},\ \abs{{\be_i^-}'}\leq C(\De\ze)^{-1},\notag
\end{align}
where C is a constant independent of $\ze_i$, $\ze_{i+1}$ and the subscript $i$. The PSTDDM consists of two source transfer steps described in Algorithm \ref{alg1} below. The two steps are independent of each other and can be computed in parallel. We recall that the construction of the original STDDM \cite{cx} is a step of source transfer followed by another step of ``wave expansion". Those two steps are executed in a sequential manner.

\begin{algorithm}
	\caption{Source Transfers for PML problem in $\R^2$}
	\label{alg1}
	\begin{multicols}{2}

		\textbf{Step 1.}

		\hrulefill

		1. Let $\bar{f}_{1}^+=f_1$;

		2. While $i=1,\cdots,N-2$ do

		\qquad$\bullet$ Find $u_i^+\in H^1(\R^2)$ such that
		\begin{align}\label{eqs1a}
			J^{-1}\na\cdot & (A\na u_i^+)+k^2u_i^+     \\
			  & =-\bar{f}_{i}^+-f_{i+1}\quad \mathrm{in}\ \R^2.\notag
		\end{align}

		\qquad$\bullet$ Compute
		\begin{align*}
			\Psi_{i+1}^+(\bar{f}_{i}^+)= & J^{-1}\na\cdot (A\na ({\be_{i+1}^+}u_i^+)) \\
			    & +k^2({\be_{i+1}^+}u_i^+).\notag
		\end{align*}

		\qquad$\bullet$ Set
		\begin{align*}
			\bar{f}_{i+1}^+=\begin{cases}
				f_{i+1}+\Psi_{i+1}^+(\bar{f}_{i}^+) \text{ in } \Omega_{i+1}, \\
				0 \text{ elsewhere. }
			\end{cases}
		\end{align*}
		\quad End while;

		3. For $i=N-1$, find $u_{N-1}^+\in H^1(\R^2)$ such that
		\begin{align}\label{eqs1b}
			J^{-1}\na\cdot & (A\na u_{N-1}^+)+k^2u_{N-1}^+    \\
			  & =-\bar{f}_{N-1}^+-f_{N}\; \mathrm{in}\ \R^2. \notag
		\end{align}

		\columnbreak

		\textbf{Step 2.}

		\hrulefill

		1. Let $\bar{f}_{N}^-=f_N$;

		2. While $i=N-1,\cdots,2$,

		\qquad$\bullet$ Find $u_i^-\in H^1(\R^2)$ such that
		\begin{align}\label{eqs2a}
			J^{-1}\na\cdot & (A\na u_i^-)+k^2u_i^-    \\
			  & =-\bar{f}_{i+1}^-\quad \mathrm{in}\ \R^2.\notag
		\end{align}
		\qquad$\bullet$ Compute
		\begin{align*}\Psi_{i}^-(\bar{f}_{i+1}^-)= & J^{-1}\na\cdot (A\na (\be_{i}^-u_i^-)) \\
			    & +k^2(\be_{i}^-u_i^-).
		\end{align*}
		\qquad$\bullet$ Set
		\begin{align*}
			\bar{f}_{i}^-=\begin{cases}
				f_{i}+\Psi_{i}^-(\bar{f}_{i+1}^-) \text{ in } \Omega_{i}, \\
				0 \text{ elsewhere. }
			\end{cases}
		\end{align*}
		\quad End while;

		3. For i=1, find $u_1^-\in H^1(\R^2)$ such that
		\begin{align}\label{eqs2b}
			J^{-1}\na\cdot & (A\na u_1^-)+k^2 u_1^-    \\
			  & =-f_1-\bar{f}_{2}^-\; \mathrm{in}\ \R^2.\notag
		\end{align}

	\end{multicols}
\end{algorithm}

By \eqref{eqs1a}, \eqref{eqs2a} and \eqref{eqs2b}, we know that $u_i^+$ is given by
\begin{align}
	 & u_i^+(x)=\int_{\Omega_i\cup\Omega_{i+1}} (\bar{f}_{i}^++f_{i+1})J(y)G(\tilde{x},\tilde{y})dy\ \forall x\in\R^2,\ i=1,\cdots,N-1,\label{u1i}
\end{align}
and $u_i^-$ is given by
\begin{align}
	 & u_i^-(x)=\int_{\Omega_{i+1}}\bar{f}_{i+1}^-(y)J(y)G(\tilde{x},\tilde{y})dy\ \forall x\in\R^2,\ i=N-1,\cdots,2, \\
	 & u_1^-(x)=\int_{\Omega_1\cup\Omega_2}(\bar{f}_{2}^-(y)+f_1(y))J(y)G(\tilde{x},\tilde{y})dy.
\end{align}

By simple calculation, we have the equivalent form of the source transfer operator $\Psi_{i+1}^+$:
\begin{align}
	\Psi_{i+1}^+(\bar{f}_{i}^+)=J^{-1}\na\cdot\left( A \na\be_{i+1}^+ u_i^+ \right) + J^{-1} {\na\be_{i+1}^+}\cdot (A\na u_i^+)-{\be_{i+1}^+}f_{i+1}\label{psi1}.
\end{align}
and it's easily obtained that $\Psi_{i+1}^+(\bar{f}_{i}^+)+{\be_{i+1}^+}f_{i+1}$ is in $L^2(\Omega_{i+1})$ and supported in $\Omega_{i+1}$.
Similarly, we can get the equivalent form for $\Psi_{i}^-$:
\begin{align}
	\Psi_{i}^-(\bar{f}_{i+1}^-)=J^{-1}\na\cdot\left( A{\na\be_{i}^-} u_i^- \right) + J^{-1} {\na\be_{i}^-}\cdot (A \na u_i^-)\label{psi2}.
\end{align}

The next two theorems verify the source transfer identities and give the integral representations of $u_i^\pm$. Let $\Ga_i:=\{x\in\R^2:\;x_1=\ze_i\}$. Given $a<b$, denote by $\Omega(a,b):=\set{x\in\R^2:\; a<x_1<b}$.

\begin{theorem}\label{th1}
	The following assertions hold:
	\begin{enumerate}
		\item[{\rm (i)}] For $i=1,\cdots,N-2$, we have, for any $x\in\Omega(\ze_{i+2}, +\infty)$,
	\end{enumerate}
	\begin{align}\label{eqt1a}
		\int_{\Omega_i} \bar{f}_{i}^+(y)\tilde{G}(x,y)dy=\int_{\Omega_{i+1}} \Psi_{i+1}^+(\bar{f}_{i}^+)(y)\tilde{G}(x,y)dy.
	\end{align}
	\begin{enumerate}
		\item[{\rm (ii)}] For the solution $u_i^+$ in \eqref{eqs1a}, we have, for any $x\in\Omega_{i+1}$, $i=1,\cdots,N-1$,
	\end{enumerate}
	\begin{align}\label{u1ia}
		u_i^+(x)=\int_{\Omega(-\infty,\ze_{i+2})} f(y)G(\tilde{x},\tilde{y})dy.
	\end{align}
\end{theorem}
\begin{proof}
	We first prove \eqref{eqt1a}. From \eqref{fs2} we have (see also \cite[(2.11)-(2.13)]{cx}, \cite[Theorem 2.8]{bp}, and \cite[Theorem 4.1]{ls01})
	$$
		\na_y\cdot(A\na_y(J^{-1}\tilde{G}(x,y)))+k^2J(J^{-1}\tilde{G}(x,y))=0,\; \forall x\in\Omega(\ze_{i+2},+\infty), y\in \Omega(\ze_1,\ze_{i+2}).
	$$
	For $x\in\Omega(\ze_{i+2},+\infty)$, $y\in\Omega_j$, $j=1,\cdots,i+1$, $\tilde{G}(x,y)$ decays exponentially as $\abs{y}\to \infty$ (cf. \cite[Lemma 2.5]{cx}). We know that $u_i^+(y)$ also decays exponentially as $\abs{y}\to\infty$ (cf. \cite[Lemma 2.6]{cx}). Therefore by integrating by parts we have
	\begin{align*}
		 & \quad\int_{\Omega_i} \bar{f}_{i}^+(y) \tilde{G}(x,y)dy          \\
		 & =-\int_{\Omega(-\infty,\ze_{i+1})} J^{-1}\big[\na_y\cdot (A\na_y u_i^+(y))+k^2Ju_i^+(y)\big] \tilde{G}(x,y)dy   \\
		 & =-\int_{\Ga_{i+1}} \left[(A\na_y u_i^+(y)\cdot e_1)J^{-1}\tilde{G}(x,y)-(A\na_y(J^{-1}\tilde{G}(x,y))\cdot e_1)u_i^+(y)\right]ds(y),
	\end{align*}
	where $e_1$ is the unit vector in the positive $x_1$-axis. By using \eqref{beta1} and integration by parts again, we have
	\begin{align*}
		\quad\int_{\Omega_i} \bar{f}_{i}^+(y) \tilde{G}(x,y)dy & = \int_{\pa \Omega_{i+1}} \big[(A\na_y ({\be_{i+1}^+}u_i^+(y))\cdot n)J^{-1}\tilde{G}(x,y) -   \\ &\qquad(A\na_y(J^{-1}\tilde{G}(x,y))\cdot n){\be_{i+1}^+}u_i^+(y)\big]ds(y)\\
		       & =\int_{\Omega_{i+1}} J^{-1}[\na_y\cdot (A\na_y ({\be_{i+1}^+}u_i^+(y)))+k^2J{\be_{i+1}^+}u_i^+(y)] \tilde{G}(x,y)dy \\
		       & =\int_{\Omega_{i+1}} \Psi_{i+1}^+(\bar{f}_{i}^+)(y)\tilde{G}(x,y)dy,
	\end{align*}
	where $n$ is the unit outer normal to $\pa \Omega_{i+1}$. That is, \eqref{eqt1a} holds.

	Since $\tilde{y}(y)=y$ and $J(y)=1$ for any $y\in B_l$, By using \eqref{u1i} and \eqref{eqt1a} we could prove \eqref{u1ia} {as follows}. For any $x\in\Omega_{i+1}$
	\begin{align*}
		u_i^+(x) & =\int_{\Omega_i\cup\Omega_{i+1}} (\bar{f}_{i}^++f_{i+1})J(y)G(\tilde{x},\tilde{y})dy      \\
		  & =\int_{\Omega_{i+1}} f_{i+1}(y) J(y) G(\tilde{x},\tilde{y})dy + \int_{\Omega_{i}} f_{i}(y) J(y) G(\tilde{x},\tilde{y})dy  \\
		  & \quad+\int_{\Omega_i} \Psi_{i}^+(\bar{f}_{i-1}^+)(y)J(y)G(\tilde{x},\tilde{y})dy       \\
		  & =\int_{\Omega_{i}\cup\Omega_{i+1}} f(y)G(\tilde{x},\tilde{y})dy+\int_{\Omega_{i-1}} \bar{f}_{i-1}^+(y)J(y)G(\tilde{x},\tilde{y})dy \\
		  & =\cdots=\int_{\cup_{j=1}^{i+1}\Omega_j}f(y)G(\tilde{x},\tilde{y})dy.
	\end{align*}
	This completes the proof of the theorem.
\end{proof}

The second step of Algorithm \ref{alg1} is similar to the first one. So by similar argument, we can obtain the following theorem whose proof is omitted.
\begin{theorem}\label{th2}
	The following assertions hold:
	\begin{enumerate}
		\item[{\rm (i)}] For $i=N-1,\cdots,2$, we have, for any $x\in\Omega(-\infty, \ze_{i})$,
	\end{enumerate}
	\begin{align}\label{eqt2a}
		\int_{\Omega_{i+1}} \bar{f}_{i+1}^-(y)\tilde{G}(x,y)dy=\int_{\Omega_{i}} \Psi_{i}^-(\bar{f}_{i+1}^-)(y)\tilde{G}(x,y)dy.
	\end{align}
	\begin{enumerate}
		\item[{\rm (ii)}] For the solution $u_i^-$, $i=N-1,\cdots,2$, in \eqref{eqs2a}, we have, for any $x\in\Omega_{i}$,
	\end{enumerate}
	\begin{align}\label{u2i}
		u_i^-(x)=\int_{\Omega(\ze_{i+1},+\infty)} f(y)G(\tilde{x},\tilde{y})dy.
	\end{align}
	\begin{enumerate}
		\item[{\rm (iii)}] For the solution $u_1^-(x)$ in \eqref{eqs2b}, we have, for any $x\in\Omega_1$,
	\end{enumerate}
	\begin{align}\label{u2ia}
		u_1^-(x)=\int_{\R^2} f(y)G(\tilde{x},\tilde{y})dy.
	\end{align}
\end{theorem}

Combining Theorem \ref{th1} and Theorem \ref{th2}, we could obtain the main result in this section.
\begin{theorem}\label{th3}
	Define $u_0^+\equiv0$ and $u_{N}^-\equiv0$. Then we have
	\begin{align*}
		\tilde{u}(x)=-(u_{i-1}^+(x)+u_{i}^-(x))\quad\text{for any } x\in\Omega_i, i=1,\cdots,N.
	\end{align*}
\end{theorem}
\begin{proof}
	From \eqref{u2ia}, it is clear that the above identity holds for $i=1$.
	For any $x\in\Omega_i$, $i=2,\cdots,N$, it follows from \eqref{u1ia}, \eqref{u2i}, and \eqref{tu} that,
	\begin{align*}
		\tilde{u}(x) & =-\int_{\R^2}f(y) G(\tilde{x},\tilde{y})dy            \\
		  & =-\left(\int_{\Omega(-\infty,\ze_{i+1})}f(y) G(\tilde{x},\tilde{y})dy+\int_{\Omega(\ze_{i+1},+\infty)}f(y)G(\tilde{x},\tilde{y})dy\right) \\
		  & =-(u_{i-1}^+(x)+u_{i}^-(x)),
	\end{align*}
	where we have used $\tilde{y}(y)=y$ in $B_l$. This completes the proof of the theorem.
\end{proof}

\subsection{The PSTDDM for the truncated PML problem} \label{sec3}
Note that the PML problems \eqref{eqs1a}--\eqref{eqs2b} in Algorithm~\ref{alg1} are defined in $\R^2$. In practice, they must be truncated into bounded domains.

We introduce local PML problems by using the PML complex coordinate stretching outside the domain $(\ze_i,\ze_{i+2})\times(-l_2,l_2)$. The PML stretching is $\tilde{x}^i(x)=(\tilde{x}_1^i(x_1),$ $\tilde{x}_2(x_2))^T$, which has been proposed in \cite{cx}, where
\begin{equation}\label{tx1}
	\tilde{x}_1^i(x_1)=x_1+\i\int_{\ze_{i+1}}^{x_1}\si_1^i(t) dt, \quad\si_1^i(t):=
	\begin{cases}
		\si_1(t+\ze_{N+1}-\ze_{i+2}) & \text{if } t>\ze_{i+1}, \\
		\si_1(t-\ze_i+\ze_1)  & \text{if } t\le\ze_{i+1}.
	\end{cases}
\end{equation}
We define
$$
	A_i(x)=\mathrm{diag}\left(\frac{\tilde{x}_2(x_2)'}{\tilde{x}_1^i(x_1)'},\frac{\tilde{x}_1^i(x_1)'}{\tilde{x}_2(x_2)'}\right),\quad J_i(x)=\tilde{x}_1^i(x_1)' \tilde{x}_2 (x_2)'.
$$
Denote by $\Omega_i^\pml=\set{x=(x_1,x_2)\in B_L:\ze_i-d_1\leq x_1\leq\ze_{i+2}+d_1}$.
The local PML problems in truncated domains can be defined for some wave source $F\in H^1(\Omega_i^\pml)'$ as: find $\phi\in H_0^1(\Omega_i^\pml)$ such that
\begin{align}\label{TPML}
	(A_i\na \phi,\na \psi) - k^2(J_i\phi,\psi)=-\pd{JF,\psi}\quad \forall \psi\in H_0^1(\Omega_i^\pml).
\end{align}
Then the PSTDDM for the PML equation in a truncated bounded domain is described in Algorithm~\ref{alg2} which is a natural modification of Algorithm~\ref{alg1}.

\begin{algorithm}
	\caption{Source Transfers for Truncated PML problem}
	\label{alg2}
	\begin{multicols}{2}

		\textbf{Step 1.}

		\hrulefill

		1. Let $\hat{f}_{1}^+=f_1$;

		2. While $i=1,\cdots,N-2$ do

		\qquad$\bullet$ Find $\hat{u}_i^+\in H_0^1(\Omega_i^{\pml})$,
		\begin{align}\label{eqs3a}
			J_i^{-1}\na\cdot & (A_i\na \hat{u}_i^+) + k^2\hat{u}_i^+ \\
			   & = - \hat{f}_{i}^+ - f_{i+1}.\notag
		\end{align}

		\qquad$\bullet$ Compute $\hat{\Psi}_{i+1}^+(\hat{f}_{i}^+)\in H^{-1}(\Omega_i^{\pml})$,
		\begin{align*} \hat{\Psi}_{i+1}^+(\hat{f}_{i}^+) = & J_i^{-1}\na\big(A_i\na ({\be_{i+1}^+}\hat{u}_i^+)\big) \\
			     & + k^2(\be_{i+1}^+\hat{u}_i^+).
		\end{align*}
		\qquad$\bullet$ Set
		\begin{align*}
			\hat{f}_{i+1}^+=\begin{cases}
				f_{i+1}+\hat{\Psi}_{i+1}^+(\hat{f}_{i}^+) \text{ in } \Omega_{i+1}\cap B_L, \\
				0 \text{ elsewhere. }
			\end{cases}
		\end{align*}
		\quad End while;

		3. For $i=N-1$, find $\hat{u}_{N-1}^+\in H_0^1(\Omega_{N-1}^{\pml})$,
		\begin{align}\label{eqs3b}
			J_{N-1}^{-1}\na\cdot & (A_{N-1} \hat{u}_{N-1}^+) + k^2\hat{u}_{N-1}^+ \\
			   & = - \hat{f}_{N-1}^+ - f_{N}.\notag
		\end{align}

		\columnbreak

		\textbf{Step 2.}

		\hrulefill

		1. Let $\hat{f}_{N}^- =f_N$;

		2. While $i=N-1,\cdots,2$,

		\qquad$\bullet$ Find $\hat{u}_i^-\in H_0^1(\Omega_i^{\pml})$,
		\begin{align}\label{eqs4a}
			J_i^{-1}\na\cdot & (A_i\na \hat{u}_i^-) + k^2\hat{u}_i^- \\
			   & = - \hat{f}_{i+1}^-.\notag
		\end{align}
		\qquad$\bullet$ Compute $\hat{\Psi}_i^-(\hat{f}_{i+1}^-)\in H^{-1}(\Omega_i^{\pml})$,
		\begin{align*}
			\hat{\Psi}_i^-(\hat{f}_{i+1}^-)= & J_i^{-1}\na(A_i\na ({\be_i^-}\hat{u}_i^-)) \\
			     & + k^2(\be_i^-\hat{u}_i^-).
		\end{align*}
		\qquad$\bullet$ Set
		\begin{align*}
			\hat{f}_i^-=\begin{cases}
				f_i+\Psi_i^-(\hat{f}_{i+1}^-) \text{ in } \Omega_i\cap B_L, \\
				0 \text{ elsewhere. }
			\end{cases}
		\end{align*}
		\quad End while;

		3. For i=1, find $\hat{u}_{1}^-\in H_0^1(\Omega_1^{\pml})$,
		\begin{align}\label{eqs4b}
			J_1^{-1}\na\cdot & (A_1\na \hat{u}_{1}^-) + k^2\hat{u}_{1}^- \\
			   & = - \hat{f}_{2}^- - f_1.\notag
		\end{align}

	\end{multicols}
\end{algorithm}

Similar to \eqref{tinfsup} (see also \cite[(3.16)]{cx}), the following $\inf$--$\sup$ condition holds for the truncated PML problem \eqref{TPML} if $\si_0d$ is sufficiently large.
\begin{align}\label{localinfsup}
	\sup_{\psi\in H_0^1(\Omega_i^\pml)} \frac{(A_i\na\phi,\na\psi)-k^2(J_i\phi,\psi)}{\norme{\psi}_{\Omega_i^\pml}}\geq \mu(k)\norme{\phi}_{\Omega_i^\pml}\quad \forall \phi\in H_0^1(\Omega_i^\pml),
\end{align}
where $\mu^{-1}(k)\leq C(\sigma_0) k^{1+\al}$ and $C(\sigma_0)$ is a constant that may polynomially depend on $\sigma_0$.

In order to estimate the PML truncation errors of Algorithm~\ref{alg2}, we introduce auxiliary functions $\bar{u}_{i}^\pm,\ i=1,\cdots,N-1$ defined as follows:
\begin{align}
	 & \bar{u}_{i}^+(x)=\int_{\Omega_i\cup\Omega_{i+1}}(\bar{f}_{i}^+(y)+f_{i+1}(y))J_i(y)G(\tilde{x}^i,\tilde{y}^i)dy,\quad i=1,\cdots,N-1,\label{eq_bup} \\
	 & \bar{u}_{i}^-(x)=\int_{\Omega_{i+1}}\bar{f}_{i+1}^-(y)J_{i}(y)G(\tilde{x}^i,\tilde{y}^i)dy,\quad i=N-1,\cdots,2,     \\
	 & \bar{u}_{1}^-(x)=\int_{\Omega_{1}\cup\Omega_2}(\bar{f}_{2}^-+f_1)J_{1}(y)G(\tilde{x}^1,\tilde{y}^1)dy. \label{eq_bum}
\end{align}
Clearly,
\begin{align}\label{barui}
	u_i^\pm=\bar{u}_{i}^\pm \quad\text{in}\quad \Omega_i\cup\Om_{i+1}.
\end{align}
\begin{lemma}\label{lem4}
	Let $\si_0d\geq1$ be sufficiently large, we have
	\begin{align*}
		\norm{\bar{u}_{i}^\pm}_{H^{\frac12}(\pa\Omega_i^\pml)}\lesssim (1+kL) e^{-\frac{1}{2}k\ga\bar{\si}}\norm{f}_{L^2(B_l)},\quad i=1,\cdots,N-1,
	\end{align*}
	where
	\begin{align}
		\ga=\min\big(\frac{d_1}{\sqrt{d_1^2+(2l_2+d_2)^2}},\frac{d_2}{\sqrt{d_2^2+(2l_1+d_1)^2}}\big),\ \bar\si=\min_{j=1,2}\int_{l_j}^{l_j+d_j}\si_j(t)\dt. \label{eq_gabsi}
	\end{align}
\end{lemma}
The proof of this lemma is omitted since it can be proved by following the proofs of \cite[Lemma~3.5--Lemma~3.6]{cx}.

The following lemma shows that $\hat f_i^\pm$ is a good approximation of $\bar f_i^\pm$ in Algorithm~\ref{alg1}.
\begin{lemma}\label{lem5}
	Let $\si_0d\geq1$ be sufficiently large. we have, for $i=2,\cdots,N-1$,
	\begin{align*}
		\norme{\bar{f}_{i}^+-\hat{f}_{i}^+}_{\Omega_i^\pml}^* & \lesssim (k\mu)^{-(i-1)} (1+kL)^2 e^{-\frac{1}{2}k\ga\bar{\si}} \norm{f}_{L^2(B_l)}, \\
		\norme{\bar{f}_{i}^--\hat{f}_{i}^-}_{\Omega_{i-1}^\pml}^* & \lesssim (k\mu)^{-(N-i)} (1+kL)^2 e^{-\frac{1}{2}k\ga\bar{\si}} \norm{f}_{L^2(B_l)}.
	\end{align*}
\end{lemma}
\begin{proof} We only prove (i) since the proof of (ii) is similar.
	From \eqref{u1i} and \eqref{eq_bup}, we have that $u_i^+(x)=\bar u_i^+(x)$ for $x\in\Omega_i\cap\Omega_{i+1}$, which implies
	\begin{align*}
		\Psi_{i+1}^+(\bar{f}_{i}^+)=J_i^{-1}\na\cdot (A_i\na ({\be_{i+1}^+}\bar u_i^+))+k^2({\be_{i+1}^+}\bar u_i^+).
	\end{align*}
	By simple calculation, we have for any $\psi\in H^1(\Omega_{i+1}^\pml)$,
	\begin{align*}
		\pd{J_i\Psi_{i+1}^+(\bar{f}_{i}^+),\psi} = -(A_i\na\be_{i+1}^+\bar u_i^+,\na\psi)_{\Omega_{i+1}} + (A_i\na\bar u_i^+, \na\be_{i+1}^+\psi)_{\Omega_{i+1}} -\pd{f_{i+1},\be_{i+1}^+\psi},
	\end{align*}
	and
	\begin{align*}
		\pd{J_i\hat\Psi_{i+1}^+(\hat{f}_{i}^+),\psi}=-(A_i\na\be_{i+1}^+\hat u_i^+,\na\psi)_{\Omega_{i+1}} + (A_i\na\hat u_i^+, \na\be_{i+1}^+\psi)_{\Omega_{i+1}} -\pd{f_{i+1},\be_{i+1}^+\psi}.
	\end{align*}
	Therefore, for $\psi\in H^1(\Omega_i^\pml)$
	\begin{align*}
		 & (J_{i-1}(\bar{f}_{i}^+-\hat{f}_{i}^+),\psi) = \big(J_{i-1}(\Psi_{i}^+(\bar{f}_{i-1}^+)-\hat{\Psi}_{i}^+(\hat{f}_{i-1}^+)), \psi\big)_{\Omega_i\cap B_L}     \\
		 & = -\big(A_{i-1}\na\be_i^+(\bar{u}_{i-1}^+-\hat{u}_{i-1}^+),\nabla \psi\big)_{\Omega_i\cap B_L} + \big(A_{i-1}\na(\bar{u}_{i-1}^+-\hat{u}_{i-1}^+),\na\be_i^+ \psi\big)_{\Omega_i\cap B_L} \\
		 & \lesssim k^{-1} \norme{\bar{u}_{i-1}^+-\hat{u}_{i-1}^+}_{\Omega_i\cap B_L} \norme{\psi}_{\Omega_i\cap B_L}          \\
		 & \lesssim k^{-1} \norme{\bar{u}_{i-1}^+-\hat{u}_{i-1}^+}_{\Omega_{i-1}^\pml} \norme{\psi}_{\Omega_i^\pml}.
	\end{align*}
	On the other hand, $\bar{u}_{i-1}^+-\hat{u}_{i-1}^+=\bar{u}_{i-1}^+$ on $\pa\Omega_{i-1}^\pml$ and for any $\psi\in H^1_0(\Omega_{i-1}^\pml)$
	\begin{align*}
		(A_{i-1}\na(\bar{u}_{i-1}^+-\hat{u}_{i-1}^+),\na\psi) - k^2(J_{i-1}(\bar{u}_{i-1}^+-\hat{u}_{i-1}^+),\psi) = \pd{J_{i-1}(\bar{f}_{i-1}^+-\hat{f}_{i-1}^+),\psi}.
	\end{align*}
	By using the inf--sup condition \eqref{localinfsup}, and Lemma \ref{lem4}, we have
	\begin{align*}
		 & \norme{\bar{u}_{i-1}^+-\hat{u}_{i-1}^+}_{\Omega_{i-1}^\pml}\lesssim \mu^{-1} \norme{\bar{f}_{i-1}^+-\hat{f}_{i-1}^+}_{\Omega_{i-1}^\pml}^*  \\
		 & \quad + \mu^{-1}(1+kL)\norm{\bar{u}_{i-1}}_{H^{\frac12}(\pa\Omega_{i-1}^\pml)}         \\
		 & \lesssim \mu^{-1} \norme{\bar{f}_{i-1}^+-\hat{f}_{i-1}^+}_{\Omega_{i-1}^\pml}^* + \mu^{-1} (1+kL)^2 e^{-\frac{1}{2}\ga\bar{\si}}\norm{f}_{L^2(B_l)}.
	\end{align*}
	Therefore,
	\begin{align*}
		\norme{\bar{f}_{i}^+-\hat{f}_{i}^+}_{\Omega_i^\pml}^* & \lesssim (k\mu)^{-1}\norme{\bar{f}_{i-1}^+-\hat{f}_{i-1}^+}_{\Omega_{i-1}^\pml}^* + (k\mu)^{-1} (1+kL)^2 e^{-\frac{1}{2}k\ga\bar{\si}}\norm{f}_{L^2(B_l)}.
	\end{align*}
	(i) follows from the induction argument and the fact that $\bar{f}_{1}^+-\hat{f}_{1}^+=0$.
	This completes the proof of the lemma.
\end{proof}

\begin{lemma}\label{lem6}
	Let $\si_0d\geq1$ be sufficiently large. For $i=1,\cdots,N-1$,
	\begin{align*}
		\norme{\bar{u}_{i}^+-\hat{u}_i^+}_{\Omega_i^\pml} & \lesssim k (k\mu)^{-i}(1+kL)^2 e^{-\frac{1}{2}k\ga\bar{\si}} \norm{f}_{L^2(B_l)}, \\
		\norme{\bar{u}_{i}^--\hat{u}_i^-}_{\Omega_{i}^\pml} & \lesssim k (k\mu)^{-(N-i)}(1+kL)^2 e^{-\frac{1}{2}k\ga\bar{\si}} \norm{f}_{L^2(B_l)}.
	\end{align*}
\end{lemma}
\begin{proof} Clearly,
	\begin{align*}
		(A_{i}\na(\bar{u}_{i}^+-\hat{u}_{i}^+),\na\psi) - k^2(J_{i}(\bar{u}_{i}^+-\hat{u}_{i}^+),\psi) = \pd{J_{i}(\bar{f}_{i}^+-\hat{f}_{i}^+),\psi}, \forall \psi\in H^1_0(\Omega_{i}^\pml).
	\end{align*}
	By using the fact that $\bar{u}_{i}^+-\hat{u}_{i}^+=\bar{u}_{i}^+$ on $\pa\Omega_{i}^\pml$, Lemmas~\ref{lem4}--\ref{lem5}, and the inf--sup condition \eqref{localinfsup}, the first result can be obtained. The second result can be proved similarly.
\end{proof}

\begin{theorem}\label{th4}
	We define $\hat{u}_{0}^+=\hat{u}_{N+1}^-=0$ in $\R^2$. Let $\hat{v}=-(\hat{u}_{i-1}^++\hat{u}_{i}^-)$ in $\Omega_i\cap B_L$ for all $i=1,2,\cdots,N$. Then for sufficiently large $\si_0d\geq1$, we have
	\begin{align}\label{eq4.1}
		\norme{\tilde{u}-\hat{v}}_{B_L}\lesssim k(k\mu)^{-(N-1)}(1+kL)^2 e^{-\frac{1}{2}k\ga\bar{\si}}\norm{f}_{L^2(B_l)},
	\end{align}
	where $\tilde u$ is the solution to the PML problem \eqref{eq_exactsol} in the whole space.
\end{theorem}
\begin{proof}
	Combining with Theorem \ref{th3}, we have $\tilde{u}(x)=-(\bar{u}_{i-1}^+(x)+\bar{u}_{i}^-(x))$ for any $x\in\Omega_i,i=1,\cdots,N$, where $\bar{u}_{0}^+=\bar{u}_{N}^-=0$ in $\R^2$. Then, by using \eqref{barui} and Lemma \ref{lem6}, we complete the proof.
\end{proof}

{\it Remark 2.2.} (i) Theorem~\ref{th4} shows that the solution $\hat v$ obtained by our PSTDDM is a fine approximation of the exact solution to the PML problem in the whole space.

(ii) The invisible constant in \eqref{eq4.1} generally polynomially depends on the size of $\Omega_i^\pml$ and the number $N$ of layers due to the inf--sup condition and the induction argument used in the proofs.

(iii) If Lemma~\ref{Linfsup} holds with $\al=0$, then \eqref{eq4.1} becomes
\begin{align}
	\norme{\tilde{u}-\hat{v}}_{B_L}\lesssim k (1+kL)^2 e^{-\frac{1}{2}k\ga\bar{\si}}\norm{f}_{L^2(B_l)},\label{eq4.1a}
\end{align}

\section{Source transfer block by block} \label{sec4}
Recall that the PSTDDM in Algorithms~\ref{alg1} and \ref{alg2} are layer-wise in which the original PML problems are decomposed into small PML problems by doing source transfers layer by layer from both positive and negative $x_1$-directions. Obviously, those small PML problems on strips may be solved by the same PSTDDM recursively but by doing source transfers along two $x_2$-directions, and as a results, they are decomposed further into small PML problems on blocks. In this section we introduce such block-wise PSTDDM by doing source transfers block by block (PSTDDMb), which further reduce the size of the subproblems and the computational costs to solve them. In this section, the generic constant is assumed to be independent of $k$, $N$, and $f$.

\subsection{The PSTDDMb for the PML problem in $\R^2$}
We only show the PSTDDM for solving the local PML equations in Algorithm~\ref{alg1}, which is described in Algorithm~\ref{alg3} below. The PSTDDMb for the original PML problem \eqref{eq_exactsol} is a combination of Algorithm~\ref{alg1} and Algorithm~\ref{alg3}. We first introduce some notation. For simplicity, we set $l_1=l_2=l,\ \bar l_1=\bar l_2=\bar l, \ d_1=d_2=d$ and assume that $l\eqsim1$ and $l<\bar l< l+d/2$.
We divide the whole space into layers:
\begin{align*}
	\Omega^1 & = \{x=(x_1,x_2)\in\R^2:x_2<\zeta_{2}\},    \\
	\Omega^j & = \{x=(x_1,x_2)\in\R^2:\zeta_j<x_2<\zeta_{j+1}\},\ j=2,\cdots,N-1, \\
	\Omega^N & = \{x=(x_1,x_2)\in\R^2:\zeta_N<x_2\}.
\end{align*}
Although $u_i^{\pm}$ in Algorithm~\ref{alg1} is defined in $\R^2$, only its restriction in $\Omega_i\cup\Omega_{i+1}$ is useful (see also Theorem~\ref{th3}). Since $\bar{u}_i^{\pm}$ defined in \eqref{eq_bup}--\eqref{eq_bum} equals to $u_i^{\pm}$ in $\Omega_i\cup\Omega_{i+1}$, the local PML problems in Algorithm~\ref{alg1} (see \eqref{eqs1a}--\eqref{eqs2b}) may be replaced by
\begin{align}\label{pmlalg1}
	J_i^{-1}\nabla \cdot (A_i\nabla \bar u_i^\pm)+k^2\bar u_i^\pm = F_i^\pm\quad \mathrm{in}\ \R^2, \quad i=1,\cdots, N-1,
\end{align}
where
\begin{align}\label{pmlalg1F}
	\begin{split}
		F_i^+&=-\bar{f}_{i}^+-f_{i+1}, \quad i=1,\cdots, N-1,\\
		F_i^-&=-\bar{f}_{i+1}^-,\quad i=2,\cdots, N-1,\quad F_1^-=-f_1-\bar{f}_{2}^-.
	\end{split}
\end{align}

Let $f_{i,j}^\pm=F_i^\pm$ in $\Omega^j$ and $ f_{i,j}^\pm=0$ elsewhere for $j=1,\cdots,N$. Denote by $\ga_i^+(x_2)$ and $\ga_i^-(x_2)$ smooth functions such that $\ga_i^+(t)=\be_i^+(t)$ and $\ga_i^-(t)=\be_i^-(t)$ for any $t\in\R$.

Similar to Algorithm~\ref{alg1}, we have the PSTDDM as shown below in Algorithm~\ref{alg3} for solving \eqref{pmlalg1}.
\begin{algorithm}
	\caption{Source Transfers for the $i^{th}$ local PML problems}
	\label{alg3}
	\begin{multicols}{2}

		\textbf{Step 1.}

		\hrulefill

		1. Let ${\bar{f}}_{i,1}^{\pm+}={f}_{i,1}^\pm$;

		2. While $j=1,\cdots,N-2$ do

		\qquad$\bullet$ Find ${u}_{i,j}^{\pm+}\in H^1(\R^2)$,
		\begin{align}\label{3.3}
			J_i^{-1}\na\cdot & (A_i\na{u}_{i,j}^{\pm+}) + k^2{u}_{i,j}^{\pm+} \\
			   & =- {\bar{f}}_{i,j}^{\pm+} - {f}_{i,j+1}^\pm,\notag
		\end{align}
		\qquad$\bullet$ Compute ${\Psi}_{i,j+1}^{\pm+}({\bar{f}}_{i,j}^{\pm+})\in H^{-1}(\R^2)$,
		\begin{align*}
			{\Psi}_{i,j+1}^{\pm+}({\bar{f}}_{i,j}^{\pm+}) = & J_i^{-1}\na\cdot\big(A_i\na ({\ga_{j+1}^+}{u}_{i,j}^{\pm+})\big) \\
			      & + k^2({\ga_{j+1}^+}{u}_{i,j}^{\pm+}),
		\end{align*}
		\qquad$\bullet$ Set
		\begin{align*}
			{\bar{f}}_{i,j+1}^{\pm+}=\begin{cases}
				{f}_{i,j+1}^\pm+{\Psi}_{i,j+1}^{\pm+}({\bar{f}}_{i,j}^{\pm+}) \text{ in } \Omega^{j+1}, \\
				0 \qquad\qquad\text{ elsewhere. }
			\end{cases}
		\end{align*}
		\quad End while

		3. Find ${u}_{i,N-1}^{\pm+}\in H^1(\R^2)$,
		\begin{align}
			J_i^{-1} \na\cdot & (A_i\na {u}_{i,N-1}^{\pm+}) + k^2{u}_{i,N-1}^{\pm+} \\
			   & = - {\bar{f}}_{i,N-1}^{\pm+} - {f}_{i,N}^\pm.\notag
		\end{align}

		\columnbreak

		\textbf{Step 2.}

		\hrulefill

		1. ${\bar{f}}_{i,N}^{\pm-}={f}_{i,N}^\pm$;

		2. While $j=N-1,\cdots,2$,

		\qquad$\bullet$ Find ${u}_{i,j}^{\pm-}\in H^1(\R^2)$,
		\begin{align}
			J_i^{-1} \na\cdot & (A_i\na{u}_{i,j}^{\pm-}) + k^2{u}_{i,j}^{\pm-} \\
			   & =-{\bar{f}}_{i,j+1}^{\pm-},\notag
		\end{align}
		\qquad$\bullet$ Compute ${\Psi}_{i,j}^{\pm-}({\bar{f}}_{i,j+1}^{\pm-})\in H^{-1}(\R^2)$,
		\begin{align*}
			{\Psi}_{i,j}^{\pm-}({\bar{f}}_{i,j+1}^{\pm-}) = & J_i^{-1}\na\cdot\big(A_i\na ({\ga_{j}^-}{u}_{i,j}^{\pm-})\big) \\
			      & + k^2({\ga_{j}^-}u_{i,j}^{-}),
		\end{align*}
		\qquad$\bullet$ Set
		\begin{align*}
			{\bar{f}}_{i,j}^-=\begin{cases}
				{f}_{i,j}^+ + {\Psi}_{i,j}^{\pm-}({\bar{f}}_{i,j+1}^{\pm-}) \text{ in } \Omega^{j}, \\
				0 \qquad\qquad \text{ elsewhere. }
			\end{cases}
		\end{align*}
		\quad End while

		3. Find ${u}_{i,1}^{\pm-}\in H^1(\R^2)$,
		\begin{align}\label{3.6}
			J_i^{-1}\na\cdot & (A_i\na u_{i,1}^{\pm-}) + k^2(u_{i,1}^{\pm-}) \\
			   & =-{\bar{f}}_{i,2}^{\pm-} - {f}_{i,1}^\pm.\notag
		\end{align}

	\end{multicols}
\end{algorithm}

By following the proof of Theorem~\ref{th3}, we may prove the following lemma.
\begin{lemma}\label{lem8}
	Let ${u}_{i,0}^{\pm+}\equiv0$ and ${u}_{i,N}^{\pm-}\equiv0$. The we have
	\begin{align*}
		{u}_i^\pm(x)=-\big({u}_{i,j-1}^{\pm+}(x)+{u}_{i,j}^{\pm-}(x)\big), \quad\mbox{for any } x\in\Omega^j,\quad 1\le i,j\le N.
	\end{align*}
\end{lemma}
By combining Theorem~\ref{th3} and Lemma~\ref{lem8} we have the following theorem which says that the solution $\tilde u$ to the original PML problem \eqref{eq_exactsol} can be expressed by $u_{i,j}^{\pm\pm}$.
\begin{theorem}\label{th5}
	Define $u_{0,j}^{+\pm}= u_{i,0}^{\pm+}=u_{N,j}^{-\pm}= u_{i,N}^{\pm-}\equiv0$. Then we have
	\begin{align*}
		\tilde{u}(x)={u}_{i-1,j-1}^{++}(x)+{u}_{i-1,j}^{+-}(x)+{u}_{i,j-1}^{-+}(x)+{u}_{i,j}^{--}(x),\quad\forall x\in \Omega_i\cap\Omega^j,\quad 1\le i,j\le N.
	\end{align*}
\end{theorem}

\subsection{The PSTDDMb for the truncated PML problem}
For the ease of presentation, we take $\sigma_1(t)$ and $\sigma_2(t)$ to be polynomials in $(-\bar l,-l)\cup(l,\bar l)$ of the form $\sigma_1(t)=\sigma_2(t)=\sigma_0\big(\frac{t-l}{\bar l-l}\big)^q$ for some $q\ge 2$. Let $\Omega_{i,j}^{\pml}=(\ze_i-d,\ze_{i+2}+d)\times(\ze_j-d,\ze_{j+2}+d)$ for $i,j=1,\cdots,N-1$ and let $\Omega_{i,j}^{\tru}=\Omega_{i,j}^\pml\cap\Omega^j$ for $j=1,\cdots,N-1$ and $\Omega_{i,N}^{\tru}=\Omega_{i,N-1}^\pml\cap\Omega^N$.

In order to truncate the problems \eqref{3.3}--\eqref{3.6}, we define the PML complex coordinate stretching $\til{x}^{i,j}=(\til{x}_1^i(x_1),\til{x}_2^j(x_2))$ outside the domain $(\ze_i,\ze_{i+2})\times(\ze_j,\ze_{j+2})$, where $\til{x}_1^i(x_1)$ is defined in \eqref{tx1} and
\begin{equation}\label{tx2}
	\tilde{x}_2^j(x_2)=x_2+\i\int_{\ze_{j+1}}^{x_2}\si_2^j(t) dt,
	\quad\si_2^j(t):=
	\begin{cases}
		\si_2(t+\ze_{N+1}-\ze_{j+2}) & \text{if } t>\ze_{j+1}, \\
		\si_2(t-\ze_j+\ze_1)  & \text{if } t\le\ze_{j+1}.
	\end{cases}
\end{equation}
Then the PML equation's coefficients are defined as
\begin{align*}
	A_{i,j}(x)=\diag\bigg(\frac{\tilde{x}_2^j(x_2)'}{\tilde{x}_1^i(x_1)'},\frac{\tilde{x}_1^i(x_1)'}{\tilde{x}_2^j(x_2)'}\bigg),\quad J_{i,j}(x)=\tilde{x}_1^i(x_1)'\tilde{x}_2^j(x_2)'.
\end{align*}

Next we solve the truncated PML problems \eqref{eqs3a}--\eqref{eqs4b} in Algorithm~\ref{alg2} by using the PSTDDM in $x_2$ directions, that is, we introduce Algorithm~\ref{alg4} to solve the following type of truncated PML problem: Find $\hat{u}_i^\pm\in H_0^1(\Omega_i^{\pml})$ such that
\begin{align}
	J_i^{-1}\na\cdot & (A_i\na {\hat u}_i^\pm) + k^2 {\hat u}_i^\pm = \check F_i^\pm,\quad i=1,\cdots,N-1.\label{checku}
\end{align}
We remark that the output $\check u_i^\pm$ of Algorithm~\ref{alg4} is an approximation of ${\hat u}_i^\pm$. Set $\check f_{i,j}^\pm=\check F_i^\pm$ in $\Omega^j$ and $ \check f_{i,j}^{\pm}=0$ elsewhere.
\begin{algorithm}
	\caption{: $\check u_i^\pm=$PSTDDM$_i(\check F_i^\pm)$.
		Source Transfers for Problem \eqref{checku}.
	}
	\label{alg4}
	\begin{multicols}{2}

		\textbf{Step 1.}

		\hrulefill

		1. Let $\hat{f}_{i,1}^{\pm+}=\check{f}_{i,1}^{\pm}$;

		2. While $j=1,\cdots,N-2$ do

		\qquad$\bullet$ Find $\hat{u}_{i,j}^{\pm+}\in H_0^1(\Omega_{i,j}^{\pml})$,
		\begin{align}\label{eqalg5-1}
			J_{i,j}^{-1}\na\cdot & (A_{i,j}\na \hat{u}_{i,j}^{\pm+}) + k^2\hat{u}_{i,j}^{\pm+} \\
			   & =-\hat{f}_{i,j}^{\pm+}-\check{f}_{i,j+1}^{\pm},\notag
		\end{align}
		\qquad$\bullet$ Compute $\hat{\Psi}_{i,j+1}^{\pm+}(\hat{f}_{i,j}^{\pm+})$, 
		\begin{align*}
			\hat{\Psi}_{i,j+1}^{\pm+}(\hat{f}_{i,j}^{\pm+}) = & J_{i,j}^{-1}\na\cdot\big(A_{i,j}\na ({\ga_{j+1}^+}\hat{u}_{i,j}^{\pm+})\big) \\
			       & + k^2({\ga_{j+1}^+}\hat{u}_{i,j}^{\pm+}),
		\end{align*}
		\qquad$\bullet$ Set
		\begin{align*}
			\hat{f}_{i,j+1}^{\pm+}=\begin{cases}
				\check{f}_{i,j+1}^{\pm}+\hat{\Psi}_{i,j+1}^{\pm+}(\hat{f}_{i,j}^{\pm+}) \text{ in } \Omega_{i,j+1}^\tru, \\
				0 \qquad\qquad\text{ elsewhere. }
			\end{cases}
		\end{align*}
		\quad End while;

		3. Find $\hat{u}_{i,N-1}^{\pm+}\in H_0^1(\Omega_{i,N-1}^{\pml})$,
		\begin{align}\label{eqalg5-2}
			J_{i,N-1}^{-1}\na\cdot & (A_{i,N-1}\na \hat{u}_{i,N-1}^{\pm+}) + k^2\hat{u}_{i,N-1}^{\pm+} \\
			   & =-\hat{f}_{i,N-1}^{\pm+} - \check{f}_{i,N}^{\pm}.\notag
		\end{align}

		\columnbreak

		\textbf{Step 2.}

		\hrulefill

		1. Let $\hat{f}_{i,N}^{\pm-}=\check{f}_{i,N}^{\pm}$;

		2. While $j=N-1,\cdots,2$,

		\qquad$\bullet$ Find $\hat{u}_{i,j}^{\pm-}\in H_0^1(\Omega_{i,j}^{\pml})$
		\begin{align}\label{eqalg6-1}
			J_{i,j}^{-1}\na\cdot & (A_{i,j}\na \hat{u}_{i,j}^{\pm-}) + k^2\hat{u}_{i,j}^{\pm-} \\
			   & = - \hat{f}_{i,j+1}^{\pm-},\notag
		\end{align}
		\qquad$\bullet$ Compute $\hat{\Psi}_{i,j}^{\pm-}(\hat{f}_{i,j+1}^{\pm-})$, 
		\begin{align*}
			\hat{\Psi}_{i,j}^{\pm-}(\hat{f}_{i,j+1}^{\pm-}) = & J_{i,j}^{-1}\na\cdot(A_{i,j}\na ({\ga_{j}^-}\hat{u}_{i,j}^{\pm-})) \\
			       & + k^2({\ga_{j}^-}\hat{u}_{i,j}^{\pm-}),
		\end{align*}
		\qquad$\bullet$ Set
		\begin{align*}
			\hat{f}_{i,j}^{\pm-}=\begin{cases}
				\check{f}_{i,j}^{\pm}+\Psi_{i,j}^{\pm-}(\hat{f}_{i,j+1}^{\pm-}) \text{ in } \Omega_{i,j}^{\tru}, \\
				0 \text{ elsewhere. }
			\end{cases}
		\end{align*}
		\quad End while;

		3. Find $\hat{u}_{i,1}^{\pm-}\in H_0^1(\Omega_{i,1}^{\pml})$,
		\begin{align}\label{eqalg6-2}
			J_{i,1}^{-1}\na\cdot & (A_{i,1}\na \hat{u}_{i,1}^{\pm-})+k^2\hat{u}_{i,1}^{\pm-} \\
			   & =-\hat{f}_{i,2}^{\pm-}-\check{f}_{i,1}^{\pm}.\notag
		\end{align}
	\end{multicols}
	\hrulefill

	Output $\check u_i^{\pm}\in H_0^1(\Omega_i^{\pml})$ by $\check u_i^{\pm} =-(\hat u_{i,j-1}^{\pm+}+\hat u_{i,j}^{\pm-})$ in $ \Omega_{i,j}^\tru$ where $\hat u_{i,0}^{\pm+}=\hat u_{i,N}^{\pm-}=0$.
\end{algorithm}

By inserting Algorithm~\ref{alg4} into Algorithm~\ref{alg2}, we obtain the PSTDDMb as described in Algorithm~\ref{alg5} for solving the truncated PML problem \eqref{eq_tpml}.
\begin{algorithm}
	\caption{PSTDDMb (Source Transfers block by block)}
	\label{alg5}
	\begin{multicols}{2}

		\textbf{Step 1.}

		\hrulefill

		1. Let $\check{f}_{1}^+=f_1$;

		2. While $i=1,\cdots,N-2$ do

		\qquad$\bullet$ Let $\check F_i^+:=- \check{f}_{i}^+ - f_{i+1}$.

		\hskip 28pt Call $\check u_i^+=$PSTDDM$_i(\check F_i^+)$.

		\qquad$\bullet$ Compute $\check{\Psi}_{i+1}^+(\check{f}_{i}^+)\in H^{-1}(\Omega_i^{\pml})$,
		\begin{align*}
			\check{\Psi}_{i+1}^+(\check{f}_{i}^+) = & J_i^{-1}\na\cdot\big(A_i\na ({\be_{i+1}^+}\check{u}_i^+)\big) \\
			     & + k^2(\be_{i+1}^+\check{u}_i^+).
		\end{align*}
		\qquad$\bullet$ Set
		\begin{align*}
			\check{f}_{i+1}^+=\begin{cases}
				f_{i+1}+\check{\Psi}_{i+1}^+(\check{f}_{i}^+) \text{ in } \Omega_{i+1}\cap B_L, \\
				0 \text{ elsewhere. }
			\end{cases}
		\end{align*}
		\quad End while;

		3. For $i=N-1$,

		\hskip 28pt Let $\check F_{N-1}^+:=- \check{f}_{N-1}^+ - f_N$.

		\hskip 28pt Call $\check u_{N-1}^+=$PSTDDM$_{N-1}(\check F_{N-1}^+)$.
		\columnbreak

		\textbf{Step 2.}

		\hrulefill

		1. Let $\check{f}_{N}^- =f_N$;

		2. While $i=N-1,\cdots,2$,

		\qquad$\bullet$ Let $\check F_i^-:=- \check{f}_{i+1}^-$.

		\hskip 28pt Call $\check u_i^-=$PSTDDM$_i(\check F_i^-)$.

		\qquad$\bullet$ Compute $\check{\Psi}_i^-(\check{f}_{i+1}^-)\in H^{-1}(\Omega_i^{\pml})$,
		\begin{align*}
			\check{\Psi}_i^-(\check{f}_{i+1}^-)= & J_i^{-1}\na\cdot(A_i\na ({\be_i^-}\check{u}_i^-)) \\
			     & + k^2(\be_i^-\check{u}_i^-).
		\end{align*}
		\qquad$\bullet$ Set
		\begin{align*}
			\check{f}_i^-=\begin{cases}
				f_i+\Psi_i^-(\check{f}_{i+1}^-) \text{ in } \Omega_i\cap B_L, \\
				0 \text{ elsewhere. }
			\end{cases}
		\end{align*}
		\quad End while;

		3. For i=1,

		\hskip 28pt Let $\check F_1^-:=- \check{f}_{2}^- - f_1$.

		\hskip 28pt Call $\check u_1^-=$PSTDDM$_1(\check F_1^-)$.
	\end{multicols}
\end{algorithm}

For further analysis we first give in the following lemma an improvement of the local inf--sup condition \eqref{localinfsup}.
\begin{lemma}\label{lem6e}
	Suppose the thickness of PML $d\eqsim l/N$. For sufficiently large $\sigma_0>1$, we have the inf--sup condition for any $\phi\in H_0^1(\Omega_{i,j}^\pml)$
	\begin{align}
		\sup_{\psi\in H_0^1(\Omega_{i,j}^\pml)} \frac{(A_{i,j}\na\phi,\na\psi)-k^2(J_{i,j}\phi,\psi)}{\norme{\psi}_{\Omega_{i,j}^\pml}}\geq \mu_N \norme{\phi}_{\Omega_{i,j}^\pml},
	\end{align}
	where $\mu_N^{-1}\leq C(\sigma_0) (1+kl/N)^{1+\al}$ and $C(\sigma_0)$ is a constant that may polynomially depend on $\sigma_0$.
\end{lemma}
\begin{proof}
	We prove the lemma by using the scaling argument. Let $l_N:=2l/N$ and introduce the affine mapping $m:\hat\Omega:=(-1-\hat d,1+\hat d)\times(-1-\hat d,1+\hat d)\rightarrow\Omega_{i,j}^\pml$ as $x=m(\hat x)=2l/N\hat x +(\ze_{i+1},\ze_{j+1})$ where $\hat d=d N/(2l)$. For any function $v$ on $\Omega_{i,j}^\pml$, let $\hat v:=v\circ m$ be the corresponding function on $\hat\Omega$. Under this affine mapping, the following PML problem
	\begin{align*}
		-\na(A_{i,j}\na\phi)-k^2J_{i,j}\phi=F\quad\text{on }H_0^1(\Omega_{i,j}^\pml)
	\end{align*}
	becomes
	\begin{align}\label{eq6e0}
		-\na_{\hat x}(\hat A_{i,j}\na_{\hat x}\hat{\phi})-(k l_N)^2\hat J_{i,j}\hat{\phi}=l_N^2\hat{F}.
	\end{align}
	If $k l_N\gtrsim 1$, from Lemma~\ref{Linfsup} we get
	\begin{align}\label{eq6e1}
		\Vert{\hskip -1pt}\vert{\hat\phi}\vert{\hskip -1pt}\Vert:=\Big(\big|\hat{\phi}\big|_{H^1(\hat\Omega)}^2+(k l_N)^2\big\|{\hat{\phi}}\big\|_{L^2(\hat\Omega)}^2 \Big)^{\frac12} \lesssim (k l_N)^{1+\al} \Vert{\hskip -1pt}\vert{l_N^2\hat{F}(z)}\Vert{\hskip -1pt}\vert^*,
	\end{align}
	under the condition that $\sigma_0\hat d>1$ is sufficiently large. Here $\norme{\cdot}^*:=\sup_{\psi\in H^1(\hat\Omega)}\frac{\pd{\cdot, \psi}}{\Vert{\hskip -1pt}\vert{\psi}\vert{\hskip -1pt}\Vert}$. On the other hand, if $k l_N$ is small enough, then it's known that the problem \eqref{eq6e0} is elliptic, and hence
	\begin{align}\label{eq6e2}
		\big\|\hat{\phi}(z)\big\|_{H^1(\hat\Omega)}\lesssim \Vert{\hskip -1pt}\vert{l_N^2\hat{F}(z)}\Vert{\hskip -1pt}\vert^*.
	\end{align}
	By combining \eqref{eq6e1}, \eqref{eq6e2}, and the following identities
	\begin{align*}
		(kl_N)^2\big\|{\hat{\phi}}\big\|_{L^2(\hat\Omega)}^2 = k^2\norm{\phi}_{L^2(\Omega_{i,j}^\pml)}^2,\;
		\big|\hat{\phi}\big|_{H^1(\hat\Omega)}^2 = \abs{\phi}_{H^1(\Omega_{i,j}^\pml)}^2,\;
		\Vert{\hskip -1pt}\vert{l_N^2\hat{F}}\Vert{\hskip -1pt}\vert^* = \norme{F}_{\Omega_{i,j}^\pml}^*,
	\end{align*}
	we have
	\begin{align*}
		\norme{\phi}_{\Omega_{i,j}^\pml}\lesssim \big(1+(k l_N)^{1+\al}\big)\norme{F}_{\Omega_{i,j}^\pml}^*.
	\end{align*}
	This completes the proof of the lemma.
\end{proof}

We remark that, by using the arguments for\cite[Lemma~3.3]{cx} and the above lemma, there also holds the following inf--sup condition for the PML equation in $\R^2$,
\begin{align}
	\sup_{\psi\in H^1(\R^2)} \frac{(A_{i,j}\na\phi,\na\psi)-k^2(J_{i,j}\phi,\psi)}{\norme{\psi}_{\R^2}} \gtrsim \mu_N \norme{\phi}_{\R^2}. \label{eq_misir2}
\end{align}
The proof is omitted.

The above inf-sup conditions indicate that if the size of blocks is small enough such that $kl/N\lesssim 1$, then the local truncated PML problems in Algorithm~\ref{alg4} (see \eqref{eqalg5-1}--\eqref{eqalg6-2}) are (nearly) elliptic which may be solved efficiently by some existent methods such as multigrid PCG method. Therefore in the following analysis we will assume that $kl/N\gtrsim 1$.

Next we introduce the auxiliary functions $\bar{u}_{i,j}^{\pm\pm}\in H^1(\R^2)$, $i,j=1,2,\cdots,N-1$ which are solutions to the PML problems:
\begin{align}
	J_{i,j}^{-1}\na\cdot(A_{i,j}\na\bar{u}_{i,j}^{\pm+}) + k^2\bar{u}_{i,j}^{\pm+} & =- {\bar{f}}_{i,j}^{\pm+} - {f}_{i,j+1}^\pm,\quad j=1,\cdots,N-1,\label{buija} \\
	J_{i,j}^{-1}\na\cdot(A_{i,j}\na\bar{u}_{i,j}^{\pm-}) + k^2\bar{u}_{i,j}^{\pm-} & =- {\bar{f}}_{i,j+1}^{\pm-},\quad j=2,\cdots,N-1,    \\
	J_{i,1}^{-1}\na\cdot(A_{i,1}\na\bar{u}_{i,1}^{\pm-}) + k^2\bar{u}_{i,1}^{\pm-} & = -\bar{f}_{i,2}^{\pm-} - f_{i,1}^\pm.\label{buijb}
\end{align}
Clearly, $\bar u_{i,j}^{\pm\pm}=u_{i,j}^{\pm\pm}$ in $\Omega^j\cup\Omega^{j+1}$ and $\hat u_{i,j}^{\pm\pm}$ in Algorithm~\ref{alg4} approximates $\bar u_{i,j}^{\pm\pm}$.

Note that the thickness of the PML in this section is allowed to be small ($d\ll 1$). The following lemma is going to be used to derive some new estimate of $\im\rho$ which is usually better than \eqref{eq_irxyi} when $d\ll 1$.
\begin{lemma}\label{lemma_imagrho}
	For any $z_1=a_1+\i b_1$, $z_2=a_2+\i b_2$ with $a_1,a_2,b_1,b_2\in\R^2$ such that $\beta=\sqrt{\frac{a_1^2+a_2^2}{b_1^2+b_2^2}}<1$, we have $\im\big((z_1^2+z_2^2)^\frac12\big)\geq\sqrt{1-\beta^2}(b_1^2+b_2^2)^\frac12$.
\end{lemma}
\begin{proof}
	Denote by $\vec{a} = (a_1,a_2)$ and $\vec{b}=(b_1,b_2)$, we have
	\begin{align*}
		\im\big((z_1^2+z_2^2)^\frac12\big) = & \sqrt{ \frac{ \abs{\big(|\vec{a}|^2-|\vec{b}|^2\big) + 2\i \vec{a}\cdot\vec{b}}- \big(|\vec{a}|^2-|\vec{b}|^2\big)}{2} } \\
		\geq     & \sqrt{ |\vec{b}|^2-|\vec{a}|^2 } \geq \sqrt{1-\beta^2} |\vec{b}|
	\end{align*}
	which completes the proof of the lemma.
\end{proof}

Lemma~\ref{lemma_imagrho} and \eqref{rho} imply that if
\begin{align}
	\abs{ \int_{y_1}^{x_1} \sigma_1(t) dt }^2 + \abs{ \int_{y_2}^{x_2} \sigma_2(t) dt }^2 \geq \frac43 \rho(x,y)^2, \label{eq_imagrhocond}
\end{align}
there holds
\begin{align}
	\im\rho(\tilde x,\tilde y) \geq \frac12 \bigg( \abs{ \int_{y_1}^{x_1} \sigma_1(t) dt }^2 + \abs{ \int_{y_2}^{x_2} \sigma_2(t) dt }^2 \bigg)^\frac12.\label{eq_imagrhoxi}
\end{align}

To proceed, we decompose $\Omega^j$ into two parts $\Omega^j=\Omega_{i+}^j\cup\Omega_{i-}^j$:
\begin{align}\label{O+-}
	\begin{split}\Omega_{i+}^j&=\{x\in\Omega^j:|x_1-\ze_{i+1}|>2l/N+d/2\},\\
		\Omega_{i-}^j&=\{x\in\Omega^j:|x_1-\ze_{i+1}|\leq 2l/N+d/2\}.
	\end{split}
\end{align}

The following two lemmas give some estimates of $\bar u_{i,j}^{\pm\pm}$.
\begin{lemma}\label{lemma_omp}
	Assume that $\sigma_0d>1$ be sufficiently large, $d\eqsim l/N$ and $ kl/N\gtrsim 1$. Then we have for $j=1,\cdots,N-1$,
	\begin{align}
		k\abs{\bar{u}_{i,j}^{\pm+}(x)} +\abs{\na\bar{u}_{i,j}^{\pm+}(x)} & \lesssim k(l/N)^{-\frac12} e^{-\frac{1}{2}k \bar{\bar{\si}}} \norm{f}_{L^2(B_l)}\quad \forall x\in\Omega_{i+}^{j+1}; \label{eq_omp_bupmp} \\
		k\abs{\bar{u}_{i,j}^{\pm-}(x)}+\abs{\na\bar{u}_{i,j}^{\pm-}(x)} & \lesssim k(l/N)^{-\frac12} e^{-\frac{1}{2}k\bar{\bar{\si}}} \norm{f}_{L^2(B_l)}\quad \forall x\in\Omega_{i+}^j,
	\end{align}
	where
	\begin{align}
		\bar{\bar\sigma} = \min\bigg( \int_{l}^{l+d/2} \sigma_1(t) dt, \int_{l}^{l+d/2} \sigma_2(t) dt \bigg).
	\end{align}
\end{lemma}
\begin{proof}
	We show only the proof of \eqref{eq_omp_bupmp} for $\bar{u}_{i,j}^{++}$. From \eqref{pmlalg1}--\eqref{pmlalg1F} and Algorithm~\ref{alg3}, similar to \eqref{u1ia}, we have, for $j=1,\cdots,N-1$,
	\begin{align} \label{eq_upp}
		\bar{u}_{i,j}^{++}(x) & = - \int_{\cup_{p=1}^{j+1}\Omega^p} f_{i+1}(y)J_{i}(y)G(\tilde{x}^i,\tilde{y}^i)dy - \int_{\cup_{p=1}^{j+1}\Omega^p} \bar{f}_i^+(y)J_{i}(y)G(\tilde{x}^i,\tilde{y}^i)dy\notag \\
		   & =:\bar{u}_{i,j}^{++\mathrm{I}}(x)+\bar{u}_{i,j}^{++\mathrm{II}}(x) \quad \forall x\in\Omega_{i+}^{j+1}.
	\end{align}

	First, noting that $\abs{x-y}\geq \abs{x_1-y_1}\geq \frac{d}2$ for $x\in\Omega_{i+}^{j+1}$ and $y\in\Om_i\cup \Om_{i+1}$, by similar arguments to \eqref{eq_imagrhocond}--\eqref{eq_imagrhoxi} we get
	\eq{\label{imrho}\im\rho(\tilde{x}^i,\tilde{y}^i)\geq \frac12\bar{\bar{\si}}\quad\text{for } x\in\Omega_{i+}^{j+1}, y\in(\ze_i,\ze_{i+2})\times(-l-d/2,l+d/2).}
	if
	\eqn{\bigg|\int_{l}^{l+\abs{x_1-\ze_{i+1}}-\frac{2l}{N}} \sigma_1(t) dt\bigg|^2\ge\frac43\Big((x_1-y_1)^2+\Big(l+\frac{d}{2}\Big)^2\Big).}
	Noting that $\big|\int_{l}^{l+\abs{x_1-\ze_{i+1}}-\frac{2l}{N}} \sigma_1(t) dt\big|\eqsim \si_0\big(\abs{x_1-\ze_{i+1}}-\frac{2l}{N}\big)\eqsim\si_0\abs{x_1-y_1}$, we conclude that \eqref{imrho} holds for sufficiently large $\sigma_0d$. Therefore, it follows from Lemma~\ref{lemma_G} and the definitions of the complex coordinate stretching $\tilde{x}^i$ and $\tilde{y}^i$ that, for $x\in\Omega_{i+}^{j+1}$
	\begin{align}\label{eq_uppi}
		\abs{\bar{u}_{i,j}^{++\mathrm{I}}(x)}
		\lesssim & \norm{f}_{L^2(\Om_{i+1}\cap B_l)}\norm{J_{i}(\cdot)G(\tilde{x}^i,\tilde{\cdot}^i)}_{L^2(\Om_{i+1}\cap B_l)}   \\
		\lesssim & \norm{f}_{L^2(\Om_{i+1}\cap B_l)}(l/N)^\frac12\max_{y\in\Om_{i+1}}|G(\tilde{x}^i,\tilde{y}^i)|\notag    \\
		\lesssim & \norm{f}_{L^2(B_l)}(l/N)^\frac12\max_{y\in\Om_{i+1}} e^{-\frac12 k\im\rho(\tilde{x}^i,\tilde{y}^i)} (kd/2)^{-\frac12}\notag \\
		\lesssim & k^{-\frac12} e^{-\frac12 k\bar{\bar\si}} \norm{f}_{L^2(B_l)}.\notag
	\end{align}

	It remains to estimate $\bar{u}_{i,j}^{++\mathrm{II}}$.	Let $\Omega_i^{j,\mathrm{in}}=\{x\in\Omega_i:\zeta_1-d/2<x_2<\ze_{j+2}\}$ for $j=1,\cdots,N-2$, and $\Omega_i^{N-1,\mathrm{in}}=\{x\in\Omega_i:\zeta_1-d/2<x_2<\ze_{N+1}+d/2\}$. Denote by $\Omega_i^{j,\mathrm{out}}=\big(\cup_{p=1}^{j+1}\Omega^p\cap\Omega_i\big)\setminus\Omega_i^{j,\mathrm{in}}$.
	It is clear that
	\begin{align}\label{eq_uijII}
		\abs{\bar{u}_{i,j}^{++\mathrm{II}}(x)} \leq \bigg| \int_{\Omega_i^{j,\mathrm{in}}}\bar{f}_i^+(y)J_i(y)G(\tilde{x}^i,\tilde{y}^i) dy\bigg| \hskip -1pt+\hskip -1pt \bigg| \int_{\Omega_i^{j,\mathrm{out}}}\bar{f}_i^+(y)J_i(y)G(\tilde{x}^i,\tilde{y}^i)dy\bigg|.
	\end{align}
	Similar to \eqref{eq_uppi}, by using \eqref{imrho} and the assumption $ kd\eqsim kl/N \gtrsim 1$, we may show that
	\begin{align} \label{eq_uppii2}
		 & \bigg|\int_{\Omega_i^{j,\mathrm{in}}}\bar{f}_i^+(y)J_{i}(y)G(\tilde{x}^i,\tilde{y}^i)dy \bigg|		 \lesssim \norme{\bar{f}_i^+}_{\Omega_i^{j,\mathrm{in}}}^*\norme{J_{i}(\cdot)G(\tilde{x}^i,\tilde{\cdot}^i)}_{\Omega_i^{j,\mathrm{in}}} \\
		 & \lesssim \norme{\bar{f}_i^+}_{\Omega_i^{j,\mathrm{in}}}^* \abs{\Omega_i^{j,\mathrm{in}}}^{\frac12} \max_{y\in\Omega_i^{j,\mathrm{in}}} \Big( k\abs{G(\tilde{x}^i,\tilde{y}^i)}+\abs{\na_yG(\tilde{x}^i,\tilde{y}^i)} \Big) \notag \\
		 & \lesssim k^{\frac12} e^{-\frac12 k \bar{\bar\si}} \norme{\bar{f}_i^+}_{\Omega_i^{j,\mathrm{in}}}^*.\notag
	\end{align}
	Next we estimate $\norme{\bar{f}_i^+}_{\Omega_i^{j,\mathrm{in}}}^*$. For any $D=(\zeta_i,\zeta_{i+1})\times(a_1,a_2)\subset\Omega_i\cap B_L$, by using Algorithm~\ref{alg1}, the definition of the source transfer operator \eqref{psi1}, and integration by parts, we conclude that
	\begin{align*}
		\norme{\bar{f}_i^+}_{D}^* & \lesssim (kl/N)^{-1} \norme{\bar u_{i-1}^+}_{D} + \norme{f_i}_{D}^*.
	\end{align*}
	Similar to \cite[Lemma 3.2]{cx} we have the following estimate whose proof is omitted.
	\eqn{		\max_{x\in\R^2} \int_{\R^2} \bigg( k^2\abs{G(\tilde{x},\tilde{y})}^2 + \abs{\na_y G(\tilde{x},\tilde{y})}^2 \bigg ) dy\lesssim k. }
	Therefore, from \eqref{barui} and \eqref{u1ia}, we have
	\begin{align}\label{eq_fip}
		\norme{\bar{f}_i^+}_{D}^*
		 & \lesssim (kl/N)^{-1}\abs{D}^\frac12\max_{x\in\R^2} \norme{G(\tilde{x},\tilde\cdot)}_{\R^2}\norm{f}_{L^2(B_l)} + \norme{f_i}_{D}^* \\
		 & \lesssim \big( (kl/N)^{-\frac12} +k^{-1}\big) \norm{f}_{L^2(B_l)} \notag        \\
		 & \lesssim (kl/N)^{-\frac12} \norm{f}_{L^2(B_l)} .\notag
	\end{align}

	In order to estimate the second term in the right hand side of \eqref{eq_uijII}, we need some estimates on $\bar{u}_{i}^+$ defined in \eqref{eq_bup}. By using Lemma~\ref{lemma_G} and \eqref{u1ia} and similar arguments as above we conclude that, for $i=1,2,\cdots,N-1$ and $y\in\Omega_i\cup\Omega_{i+1}$ satisfying $\abs{y_2}>l+d/2$,
	\begin{align}
		\abs{\bar{u}_{i}^+(y)} \lesssim & k(kd)^{-\frac12} e^{-\frac12 k\bar{\bar\si}}\norme{f}_{B_l}^*, \label{eq_baruip1} \\
		\abs{\na\bar{u}_{i}^+(y)} \lesssim & k^2(kd)^{-\frac12} e^{-\frac12 k\bar{\bar\si}}\norme{f}_{B_l}^*. \label{eq_baruip2}
	\end{align}
	We also need estimates of two integrals on ${\Omega_i^{j,\mathrm{out}}}$ involves the Green function $G$. Noting that $|x_m-y_m|\geq d/2\ (m=1,2)$, from \eqref{eq_irxyi}, we have, ${\mathrm{Im}}\rho(\tilde x^i,\tilde y^i) \ge 0$. Moreover, from \cite[Lemma 3.1]{cx} (with $\be=2$ there), we have
	\begin{align*}
		{\mathrm{Im}}\rho(\tilde x^i,\tilde y^i)\geq\frac14\sigma_0|x-y|\ge \frac14\sigma_0|x_2-y_2| \quad\mathrm{if}\ |x-y|\geq 4\sqrt{2} M,
	\end{align*}
	where $M=l+d$.
	Therefore, from Lemma~\ref{lemma_G} we have
	\begin{align*}
		  & \int_{\Omega_i^{j,\mathrm{out}}} \big| G(\tilde x^i,\tilde y^i) \big| dy                        \\
		\lesssim & k^{-1/2} \bigg( \int_{\Omega_i^{j,\mathrm{out}}\qquad\quad\atop |x_2-y_2|<4\sqrt{2}M} \frac{1}{|x_2-y_2|^\frac12}dy + \int_{\Omega_i^{j,\mathrm{out}}\qquad\quad\atop |x_2-y_2|\ge 4\sqrt{2}M} \frac{e^{-\frac18 k\sigma_0|x_2-y_2|}}{|x_2-y_2|^\frac12}dy \bigg) \\
		\lesssim & \frac{l}{N} k^{-\frac12}\Big(1 +(k\sigma_0)^{-\frac12} e^{-\frac14\sqrt{2}M k\sigma_0}\Big),
	\end{align*}
	which implies that
	\begin{align}
		\int_{\Omega_i^{j,\mathrm{out}}} \big| G(\tilde x^i,\tilde y^i) \big| dy\lesssim k^{-\frac12}l/N\quad \forall x\in\Omega_{i+}^{j+1},\label{eq_gi1}
	\end{align}
	if $k\sigma_0\gtrsim 1$. Similarly, for $k\sigma_0\gtrsim 1$, we have
	\begin{align}
		\int_{\Omega_i^{j,\mathrm{out}}} \big| \nabla_y G(\tilde x^i,\tilde y^i) \big| dy\lesssim k^{\frac12}l/N\quad \forall x\in\Omega_{i+}^{j+1}.\label{eq_gi2}
	\end{align}
	Therefore, by using Algorithm~\ref{alg1}, \eqref{psi1}, \eqref{beta1} and \eqref{eq_baruip1}--\eqref{eq_gi2},
	\begin{align}\label{eq_uijIIb}
		  & \bigg| \int_{\Omega_i^{j,\mathrm{out}}}\bar{f}_i^+(y) J_i(y) G(\tilde{x}^i,\tilde{y}^i)dy \bigg|             \\
		= & \bigg| \int_{\Omega_i^{j,\mathrm{out}}} \big(f_i(y)+\Psi_i^+(\bar f_{i-1}^+)(y)\big) J_i(y) G(\tilde{x}^i,\tilde{y}^i)dy \bigg|  \notag       \\
		= & \bigg| \int_{\Omega_i^{j,\mathrm{out}}} f_i(y) J_i(y) G(\tilde{x}^i,\tilde{y}^i)dy- \int_{\Omega_i^{j,\mathrm{out}}} \be_i^+(y_1) f_i(y) J_i(y) G(\tilde{x}^i,\tilde{y}^i)dy  \notag \\
		  & - \int_{\Omega_i^{j,\mathrm{out}}} \frac{\tilde x_2^i(y_2)'}{\tilde x_1^i(y_1)'}(\be_i^{+})'(y_1)\bar{u}_{i-1}^+(y)\frac{\pa G(\tilde{x}^i,\tilde{y}^i)}{\pa y_1} dy \notag   \\
		  & + \int_{\Omega_i^{j,\mathrm{out}}} (\be_i^+)'(y_1) \frac{\tilde x_2^i(y_2)'}{\tilde x_1^i(y_1)'} \frac{\pa \bar{u}_{i-1}^+(y)}{\pa y_1} G(\tilde{x}^i,\tilde{y}^i) dy \bigg| \notag  \\
		\lesssim & (l/N)^{-1}\bigg( \max_{y\in\Omega_i^{j,\mathrm{out}}} \abs{\bar{u}_{i-1}^+(y)} \int_{\Omega_i^{j,\mathrm{out}}}\abs{\frac{\pa G(\tilde{x}^i,\tilde{y}^i)}{\pa y_1}}dy\notag   \\
		  & + \max_{y\in\Omega_i^{j,\mathrm{out}}}\abs{\na\bar{u}_{i-1}^+(y)} \int_{\Omega_i^{j,\mathrm{out}}}\abs{G(\tilde{x}^i,\tilde{y}^i)}dy \bigg)\notag       \\
		\lesssim & k^\frac32 (kd)^{-\frac12} e^{-\frac12 k\bar{\bar\si}} \norme{f}_{B_l}^* \notag               \\
		\lesssim & (l/N)^{-\frac12} e^{-\frac12 k\bar{\bar\si}} \norm{f}_{L^2(B_l)},     \notag
	\end{align}
	where we have used the fact that $f_i=0$ in $\Omega_i^{j,\mathrm{out}}$.

	Thus by combining \eqref{eq_uijII}--\eqref{eq_fip}, and\eqref{eq_uijIIb}, we conclude that, for any $x\in\Omega_{i+}^{j+1}$,
	\begin{align} \label{eq_uppii}
		\abs{\bar{u}_{i,j}^{++\mathrm{II}}(x)} \lesssim & (l/N)^{-\frac12} e^{-\frac12 k\bar{\bar\si}} \norm{f}_{L^2(B_l)}.
	\end{align}

	Combining \eqref{eq_upp}, \eqref{eq_uppi}, \eqref{eq_uppii}, and $k\gtrsim1$ yields that, for $x\in\Omega_{i+}^{j+1}$,
	\begin{align}\label{eq_bupp}
		\abs{\bar{u}_{i,j}^{++}(x)} \lesssim & (l/N)^{-\frac12} e^{-\frac{1}{2}k\bar{\bar\si}} \norm{f}_{L^2(B_l)}.
	\end{align}

	By similar arguments as above we have
	\begin{align}
		\abs{\na\bar{u}_{i,j}^{++}(x)} \lesssim k(l/N)^{-\frac12} e^{-\frac{1}{2}k\bar{\bar\si}} \norm{f}_{L^2(B_l)}. \label{eq_nbupp}
	\end{align}
	Finally, \eqref{eq_omp_bupmp} follows by combining \eqref{eq_bupp}--\eqref{eq_nbupp}.
\end{proof}

\begin{lemma}\label{lem9}
	Assume that $\sigma_0d>1$ is sufficiently large, $d\eqsim l/N$ and $ kl/N\gtrsim 1$. Then there exists a constant $c_0>0$ independent of $k,l,$ and $N$ such that for $j=1,\cdots,N-1$, we have
	\begin{align*}
		\norm{\bar{u}_{i,j}^{\pm\pm}}_{H^{\frac12}(\pa\Omega_{i,j}^\pml)}\lesssim C_{k,l,N} e^{-\frac{1}{2}k\bar{\bar\si}} \norm{f}_{L^2(B_l)},
	\end{align*}
	where $C_{k,l,N} := (kl/N)^{\frac32} (c_0\mu_N kl/N)^{-(N-2)}$.
\end{lemma}

\begin{proof}
	We only sketch the proof for $\bar u_{i,j}^{++}$. From \eqref{eq_vhga12} and \cite[(3.1)]{cx}, it is obvious that
	\begin{align}
		  & \norm{\bar{u}_{i,j}^{++}}_{H^{\frac12}(\pa\Omega_{i,j}^\pml)} \label{eq_uijpml}                   \\
		\leq & \big(\abs{\pa\Omega_{i,j}^\pml}d_{\Omega_{i,j}^\pml}^{-1} \big)^{\frac12} \norm{\bar{u}_{i,j}^{++}}_{L^\infty(\pa\Omega_{i,j}^\pml)} + \abs{\pa\Omega_{i,j}^\pml} \norm{\na\bar{u}_{i,j}^{++}}_{L^\infty(\pa\Omega_{i,j}^\pml)}\notag \\
		\lesssim & \norm{\bar{u}_{i,j}^{++}}_{L^\infty(\pa\Omega_{i,j}^\pml)} + (d+l/N) \norm{\na\bar{u}_{i,j}^{++}}_{L^\infty(\pa\Omega_{i,j}^\pml)}.\notag
	\end{align}

	From \eqref{buija} and Algorithm~\ref{alg3} we have
	\eqn{\bar{u}_{i,j}^{++}&(x) = \int_{\mathbb{R}^2} (\bar f_{i,j}^{++}(y)+f_{i,j+1}^+(y)) J_{i,j} G(\tilde{x}^{i,j},\tilde{y}^{i,j}) dy \\
	=& \int_{\mathbb{R}^2} ( f_{i,j}^{+}(y)+f_{i,j+1}^+(y)) J_{i,j} G(\tilde{x}^{i,j},\tilde{y}^{i,j}) dy+\int_{\Omega^j} \Psi_{i,j}^{++}({\bar{f}}_{i,j-1}^{++}) J_{i,j} G(\tilde{x}^{i,j},\tilde{y}^{i,j}) dy
	}
	Similar to \eqref{psi1}, some simple calculations yield
	\eq{\label{Psi++}\Psi_{i,j}^{++}({\bar{f}}_{i,j-1}^{++})=&J_i^{-1}\na\cdot\left( A_i \na\ga_j^+ u_{i,j-1}^{++} \right) \\
	&+ J_i^{-1} {\na\ga_j^+}\cdot (A_i\na u_{i,j-1}^{++})-{\ga_j^+}f_{i,j}^+
	\quad\text{in }\Omega^j.\notag}
	Noting that $\bar{u}_{i,j-1}^{++}(y)$ and $\frac{\pa\bar{u}_{i,j-1}^{++}(y)}{\pa y_2}$ decay exponentially in $\Omega^j$ as $|y|\rightarrow\infty$ and that $J_{i,j}=J_i, A_{i,j}=A_i$ in $\Omega^j$, by using integration by parts, we have that, for $x\in\pa\Omega_{i,j}^\pml$
	\begin{align*}
		\bar{u}_{i,j}^{++}(x) & = \int_{\mathbb{R}^2} \big((1-\ga_j^+) f_{i,j}^+(y)+f_{i,j+1}^+(y)\big) J_{i,j} G(\tilde{x}^{i,j},\tilde{y}^{i,j}) dy       \\
		   & \quad- \int_{\Omega^j} \frac{\tilde{x}_1^i(x_1)'}{\tilde{x}_2^j(x_2)'} {\ga_j^{+}}'(y_2) \bar{u}_{i,j-1}^{++}(y) \frac{\pa G(\tilde{x}^{i,j},\tilde{y}^{i,j})}{\pa y_2} dy \\
		   & \quad+ \int_{\Omega^j} \frac{\tilde{x}_1^i(x_1)'}{\tilde{x}_2^j(x_2)'} {\ga_j^+}'(y_2) \frac{\bar{u}_{i,j-1}^{++}(y)}{\pa y_2} G(\tilde{x}^{i,j},\tilde{y}^{i,j}) dy \\
		   & := I + II + III.
	\end{align*}

	From \eqref{O+-}, it is clear that
	\begin{align*}
		|II| \lesssim (l/N)^{-1} \bigg( & \int_{\Omega_{i-}^j} \abs{\bar{u}_{i,j-1}^{++}(y)} \abs{\na_y G(\tilde{x}^{i,j},\tilde{y}^{i,j})} dy \\&+ \int_{\Omega_{i+}^j} \abs{\bar{u}_{i,j-1}^{++}(y)} \abs{\na_y G(\tilde{x}^{i,j},\tilde{y}^{i,j})} dy \bigg).
	\end{align*}
	By \eqref{buija}, the inf--sup condition \eqref{eq_misir2}, and \eqref{Psi++}, we conclude that
	\begin{align*}
		\norme{\bar u_{i,j-1}^{++}}_{\R^2} \leq & \mu_N^{-1} \norme{\bar f_{i,j-1}^{++} + f_{i,j}^{+}}_{\R^2}^*          \\
		\leq     & \mu_N^{-1} \big( \norme{\Psi_{i,j-1}^{++}({\bar f}_{i,j-2}^{++})}_{\Omega^{j-1}}^* + \norme{F_{i}^{+}}_{\Omega^{j-1}\cup\Omega^{j}}^* \big) \\
		\leq     & (c_0\mu_N kl/N)^{-1} \norme{\bar u_{i,j-2}^{++}}_{\Omega^{j-1}} + (c_0\mu_N)^{-1}\norme{\bar f_i^++f_{i+1}}_{\Omega^{j-1}\cup\Omega^{j}}^*
	\end{align*}
	where $\bar u_{i,0}^{++}=0$ and $c_0>0$ is a constant independent of $k,l,$ and $N$. Without loss of generality, we assume that $c_0\mu_N kl/N<1$. By recursion and \eqref{eq_fip}, we have
	\begin{align*}
		\norme{\bar u_{i,j-1}^{++}}_{\R^2}\leq & (c_0\mu_N)^{-1} \sum_{p=0}^{j-2} (c_0\mu_N kl/N)^{-p} \norme{\bar f_i^++f_{i+1}}_{\Omega^{j-1-p}\cup\Omega^{j-p}}^* \\			\lesssim & (c_0\mu_N)^{-1} \sum_{p=0}^{j-2} (c_0\mu_N kl/N)^{-p} (kl/N)^{-\frac12} \norm{f}_{L^2(B_l)}     \\
		\lesssim    & \mu_N^{-1} (kl/N)^{-\frac12} \sum_{p=0}^{j-2} (c_0\mu_N kl/N)^{-p} \norm{f}_{L^2(B_l)}
	\end{align*}
	Since $|x_1-y_1|\geq d/2$ or $|x_2-y_2|\geq d/2$ for $x\in\pa\Omega_{i,j}^\pml$ and $y\in \Omega_{i-}^j$, similar to \eqref{imrho}, from \eqref{eq_imagrhocond}--\eqref{eq_imagrhoxi}, we conclude that
	$$\im\rho(\tilde{x}^{i,j},\tilde{y}^{i,j})\ge\frac12 \bar{\bar\si}$$ if $ \bar{\bar\si}$ is sufficiently large. Therefore by combining Lemma~\ref{lemma_G} and the above two inequalities, we have
	\begin{align*}
		 & \int_{\Omega_{i-}^j} \abs{\bar{u}_{i,j-1}^{++}(y)} \abs{\na_y G(\tilde{x}^{i,j},\tilde{y}^{i,j})} dy
		\leq \norm{\bar{u}_{i,j-1}^{++}}_{L^2(\Omega_{i-}^j)} \cdot \bigg(\int_{\Omega_{i-}^j} \abs{\na_y G(\tilde{x}^{i,j},\tilde{y}^{i,j})}^2 dy\bigg)^{\frac12} \\
		 & \lesssim k^{-1} \mu_N^{-1} (kl/N)^{-\frac12} \sum_{p=0}^{j-2}(c_0\mu_N kl/N)^{-p} \norm{f}_{L^2(B_l)}       \\
		 & \quad \cdot (l/N) \big( k(kd/2)^{-\frac12} +k (kd/2)^{-1} \big) e^{-\frac12 k\im\rho(\tilde{x}^{i,j},\tilde{y}^{i,j})}     \\
		 & \lesssim (k\mu_N)^{-1} \sum_{p=0}^{j-2} (c_0\mu_N kl/N)^{-p} e^{-\frac12 k \bar{\bar\si}} \norm{f}_{L^2(B_l)}.
	\end{align*}
	By similar arguments to \eqref{eq_gi1}--\eqref{eq_gi2}, we obtain
	\begin{align}
		\int_{\Omega_{i+}^j} \abs{\na_y G(\tilde{x}^{i,j},\tilde{y}^{i,j})} dy \lesssim k^\frac12l/N,
	\end{align}
	which together with Lemma~\ref{lemma_omp} implies that
	\begin{align*}
		\int_{\Omega_{i+}^j} \abs{\bar{u}_{i,j-1}^{++}(y)} \abs{\na_y G(\tilde{x}^{i,j},\tilde{y}^{i,j})} dy & \leq \max_{y\in\Omega_{i+}^j} \abs{\bar{u}_{i,j-1}^{++}(y)} \int_{\Omega_{i+}^j} \abs{\na_y G(\tilde{x}^{i,j},\tilde{y}^{i,j})} dy \\
		             & \lesssim (kl/N)^{\frac12} e^{-\frac12 k \bar{\bar\si}} \norm{f}_{L^2(B_l)}.
	\end{align*}
	Therefore, we have for any $x\in\pa\Omega_{i,j}^\pml$
	\begin{align}\label{eq_pmlii}
		\abs{II} \lesssim \Big( (k\mu_N)^{-1} \sum_{p=0}^{j-2}(c_0\mu_N kl/N)^{-p} +(kl/N)^{\frac12} \Big) e^{-\frac12 k \bar{\bar\si}} \norm{f}_{L^2(B_l)}.
	\end{align}

	By similar arguments, we have
	\begin{align}
		\abs{III} \lesssim \Big( (k\mu_N)^{-1} \sum_{p=0}^{j-2} (c_0\mu_N kl/N)^{-p} +(kl/N)^{\frac12} \Big) e^{-\frac12 k \bar{\bar\si}} \norm{f}_{L^2(B_l)}.
	\end{align}

	It remains to estimate $I$. Denote by $\Omega_{i,j}=(\ze_i,\ze_{i+2})\times(\ze_j,\ze_{j+2})$. By the definition of $f_{i,j}^+$ (see the line below \eqref{pmlalg1F}), we derive in analogy to the proof of Lemma~\ref{lemma_omp} (cf. \eqref{eq_fip}), for $x\in\partial\Omega_{i,j}^\pml$
	\begin{align}
		\abs{I} \leq & \norme{ f_{i,j}^++f_{i,j+1}^+ }_{\Omega_{i,j}}^* \norme{ J_{i,j}(\cdot)G(\tilde x^{i,j},\tilde\cdot) }_{\Omega_{i,j}} \label{eq_pmli}          \\
		\leq  & \norme{ f_{i,j}^++f_{i,j+1}^+ }_{\Omega_{i,j}}^* \cdot \abs{\Omega_{i,j}}^{\frac12} \max_{y\in\Omega_{i,j}}\Big( k\abs{G(\tilde x^{i,j},\tilde y^{i,j})} + \abs{\nabla G(\tilde x^{i,j},\tilde y^{i,j})} \Big) \notag \\
		\lesssim & \norme{ f_{i,j}^++f_{i,j+1}^+ }_{\Omega_{i,j}}^* \cdot l/N e^{ -\frac12 k\im\rho(\tilde x^{i,j},\tilde y^{i,j}) }\big[ k(kd/2)^{-\frac12}+ k(kd/2)^{-1} \big] \notag       \\
		\lesssim & (kl/N)^{-\frac12}\norm{f}_{L^2(B_l)} \cdot (kl/N)^{\frac12} e^{-\frac{1}{2}k\bar{\bar\si}} \notag               \\
		\lesssim & e^{-\frac{1}{2}k\bar{\bar\si}} \norm{f}_{L^2(B_l)}. \notag
	\end{align}

	By combining \eqref{eq_pmlii}--\eqref{eq_pmli}, we obtain that, for $x\in\partial\Omega_{i,j}^\pml$,
	\begin{align}\label{eq_uijabs}
		\abs{\bar{u}_{i,j}^{++}(x)} \lesssim & \Big( (k\mu_N)^{-1} \sum_{p=0}^{j-2} (c_0\mu_N kl/N)^{-p} +(kl/N)^{\frac12} \Big) e^{-\frac{1}{2}k\bar{\bar\si}} \norm{f}_{L^2(B_l)} \\
		\lesssim    & (kl/N)^{\frac12} (c_0\mu_N kl/N)^{-j+1} e^{-\frac{1}{2}k\bar{\bar\si}} \norm{f}_{L^2(B_l)}.\notag
	\end{align}
	where we have used $ (k\mu_N)^{-1} \lesssim k^{-1}(kl/N)^\frac32\lesssim (kl/N)^\frac12$ to derive the last inequality (see Lemma~\ref{lem6e} and Remark 2.1 (i)).

	By similar arguments, we have for $x\in\partial\Omega_{i,j}^\pml$,
	\begin{align}\label{eq_nauijabs}
		\abs{\nabla\bar{u}_{i,j}^{++}(x)} \lesssim k (kl/N)^{\frac12} (c_0\mu_N kl/N)^{-j+1} e^{-\frac{1}{2}k\bar{\bar\si}} \norm{f}_{L^2(B_l)}.
	\end{align}

	Finally, we complete the proof by combining \eqref{eq_uijpml}, \eqref{eq_uijabs}--\eqref{eq_nauijabs}.
\end{proof}

\begin{lemma}\label{lem7}
	Let $\sigma_0>1$ be sufficiently large. Assume that $d\eqsim l/N$ and $kl/N\gtrsim 1$. There exists a constant $\tilde C$ independent of $f$, $k$, and $N$ such that for $i=1,2,\cdots,N-1$, we have
	\begin{align}\label{eq7a}
		\norme{\check{u}_i^\pm-\bar{u}_i^\pm}_{\Omega_i^{\pml}} \lesssim \tilde C_{k,l,N} e^{-\frac{1}{2}k\bar{\bar\si}} \norm{f}_{L^2(B_l)},
	\end{align}
	where $\tilde C_{k,l,N}= (kl/N)^\frac72 ( \tilde C \mu_Nkl/N)^{-(N^2+2N-2)}$.
\end{lemma}

\begin{proof}
	We only prove \eqref{eq7a} for $\norm{\check u_i^+-\bar u_i^+}_{H^1(\Omega_i^\pml)}$. Noting from \eqref{barui} that $\bar{u}_i^{\pm}=u_i^{\pm}$ in $\Omega_i\cup\Omega_{i+1}$, from \eqref{pmlalg1}--\eqref{pmlalg1F}, Algorithms~\ref{alg1} and \ref{alg5}, and some simple calculations, we conclude that (cf. \eqref{psi1})
	\begin{align*}
		\Psi_i^+ & = J^{-1}\nabla\cdot (A_{i-1}\nabla \beta_i^+\bar u_{i-1}^+)+J^{-1}\nabla \beta_i^+\cdot (A_{i-1}\nabla\bar u_{i-1}^+)-\beta_i^+ f_i\quad \text{in } \Omega_i, \\
		\check\Psi_i^+ & = J^{-1}\nabla\cdot (A_{i-1}\nabla \beta_i^+\check u_{i-1}^+)+J^{-1}\nabla \beta_i^+\cdot (A_{i-1}\nabla\check u_{i-1}^+)-\beta_i^+ f_i\quad \text{in } \Omega_i,
	\end{align*}
	which implies that
	\begin{align*}
		\norme{F_{i}^+-\check{F}_{i}^+}_{\Omega_{i,j}^\mathrm{PML}}^* \lesssim \big(kl/N\big)^{-1} \norme{\bar{u}_{i-1}^+-\check{u}_{i-1}^+}_{\Omega_i\cap\Omega_{i,j}^\pml},\quad i=1,\cdots,N-1,
	\end{align*}
	where $\bar{u}_0^+-\check{u}_0^+=0$.
	Similarly, from Algorithm \ref{alg4}, we have,
	\eq{\label{hPsi++}\hat\Psi_{i,j}^{++}({\hat{f}}_{i,j-1}^{++})=&J_i^{-1}\na\cdot\left( A_i \na\ga_j^+ \hat u_{i,j-1}^{++} \right) \\
	&+ J_i^{-1} {\na\ga_j^+}\cdot (A_i\na \hat u_{i,j-1}^{++})-{\ga_j^+}\check f_{i,j}^+
	\quad\text{in }\Omega_{i,j}^{\tru},\notag}
	which together with \eqref{Psi++} implies that, for $i,j=1,\cdots,N-1,$
	\begin{align*}
		\norme{\bar{f}_{i,j}^{++}-\hat{f}_{i,j}^{++}}_{\Omega_{i,j}^\pml}^* & \lesssim \big(kl/N\big)^{-1} \norme{\bar{u}_{i,j-1}^{++}-\hat{u}_{i,j-1}^{++}}_{\Omega_{i,j-1}^\pml} +\norme{F_{i}^+-\check{F}_{i}^+}_{\Omega_{i,j}^\mathrm{PML}}^*  \\
		\lesssim        & \big(kl/N\big)^{-1} \Big(\norme{\bar{u}_{i,j-1}^{++}-\hat{u}_{i,j-1}^{++}}_{\Omega_{i,j-1}^\pml} +\norme{\bar{u}_{i-1}^+-\check{u}_{i-1}^+}_{\Omega_i\cap\Omega_{i,j}^\pml}\Big),
	\end{align*}
	where $\bar{u}_{i,0}^{++}-\hat{u}_{i,0}^{++}=0$.

	In the following arguments, we denote by $\tilde C$ the generic constant independent of $k,l/N,N$ and $\mu_N$. By combining the local inf--sup condition~\ref{lem6e}, the inverse trace inequality, the scaling argument, and Lemma~\ref{lem9}, we obtain
	\begin{align}
		 & \norme{\bar{u}_{i,j}^{++}-\hat{u}_{i,j}^{++}}_{\Omega_{i,j}^\pml} \lesssim \mu_N^{-1} \norme{(\bar{f}_{i,j}^{++}+{f}_{i,j+1}^+)-(\hat{f}_{i,j}^{++}+\check{f}_{i,j+1}^+)}_{\Omega_{i,j}^\pml}^* \label{eq7.1} \\
		 & \quad + (1+\mu_N^{-1}) (1+kl/N) \norm{\bar{u}_{i,j}^{++}}_{H^{\frac12}(\pa\Omega_{i,j}^\pml)} \notag              \\
		 & \lesssim \mu_N^{-1}(kl/N)^{-1}\Big( \norme{\bar{u}_{i,j-1}^{++}-\hat{u}_{i,j-1}^{++}}_{\Omega_{i,j-1}^\pml}+\norme{\bar{u}_{i-1}^+-\check{u}_{i-1}^+}_{\Omega_i\cap\Omega_{i,j}^\pml} \Big) \notag  \\
		 & \quad+ \mu_N^{-1}\norme{F_{i}^+-\check{F}_{i}^+}_{\Omega_{i,j}^\pml}^* + C_{k,l,N}\mu_N^{-1}kl/N e^{-\frac{1}{2}k\bar{\bar\si}} \norm{f}_{L^2(B_l)} \notag       \\
		 & \leq ( \tilde C \mu_Nkl/N)^{-1} \norme{\bar{u}_{i,j-1}^{++}-\hat{u}_{i,j-1}^{++}}_{\Omega_{i,j-1}^\pml} +( \tilde C \mu_Nkl/N)^{-1} \notag         \\
		 & \quad \cdot \norme{\bar{u}_{i-1}^+-\check{u}_{i-1}^+}_{\Omega_i\cap\Omega_{i,j}^\pml} + C_{k,l,N}(\tilde C\mu_N)^{-1}kl/N e^{-\frac{1}{2}k\bar{\bar\si}} \norm{f}_{L^2(B_l)},\notag
	\end{align}
	for $i,j=1,\cdots,N-1$ if $\sigma_0$ is sufficiently large. By induction,
	\begin{align}\label{eq7.2}
		 & \norme{\bar{u}_{i,j}^{++}-\hat{u}_{i,j}^{++}}_{\Omega_{i,j}^\pml} \leq \sum_{p=1}^j ( \tilde C \mu_Nkl/N)^{-p} \cdot \norme{\bar{u}_{i-1}^+-\check{u}_{i-1}^+}_{\Omega_i\cap B_L} \\
		 & \qquad + C_{k,l,N}(\tilde C\mu_N)^{-1}kl/N \sum_{p=0}^{j-1}( \tilde C \mu_Nkl/N)^{-p} \cdot e^{-\frac{1}{2}k\bar{\bar\si}} \norm{f}_{L^2(B_l)}\notag    \\
		 & \quad\leq ( \tilde C \mu_Nkl/N)^{-(j+1)}\cdot \norme{\bar{u}_{i-1}^+-\check{u}_{i-1}^+}_{\Omega_i\cap B_L} \notag         \\
		 & \qquad + C_{k,l,N}(kl/N)^2 ( \tilde C \mu_Nkl/N)^{-(j+1)} \cdot e^{-\frac{1}{2}k\bar{\bar\si}} \norm{f}_{L^2(B_l)}. \notag
	\end{align}

	Similarly, we have for $j=N-1,\cdots,1$,
	\begin{align}
		 & \norme{\bar{u}_{i,j}^{+-}-\hat{u}_{i,j}^{+-}}_{\Omega_{i,j}^\pml} \leq \sum_{p=1}^{N-j} (\tilde C \mu_Nkl/N)^{-p} \cdot \norme{\bar{u}_{i-1}^+-\check{u}_{i-1}^+}_{\Omega_i\cap B_L} \label{eq7.3} \\
		 & \qquad + C_{k,l,N}(\tilde C\mu_N)^{-1}kl/N \sum_{p=0}^{N-j-1} (\tilde C \mu_Nkl/N)^{-p} \cdot e^{-\frac{1}{2}k\bar{\bar\si}} \norm{f}_{L^2(B_l)} \notag      \\
		 & \quad \leq ( \tilde C \mu_Nkl/N)^{-(N-j+1)}\cdot \norme{\bar{u}_{i-1}^+-\check{u}_{i-1}^+}_{\Omega_i\cap B_L} \notag          \\
		 & \qquad + C_{k,l,N} (kl/N)^2 ( \tilde C \mu_Nkl/N)^{-(N-j+1)} \cdot e^{-\frac{1}{2}k\bar{\bar\si}} \norm{f}_{L^2(B_l)}. \notag
	\end{align}

	Combining \eqref{eq7.2}, \eqref{eq7.3}, Algorithm~\ref{alg4}, Lemma~\ref{lem8}, and the definition of $C_{k,l,N}$ in Lemma~\ref{lem9}, we conclude that
	\begin{align}
		 & \norme{\bar{u}_{i}^+-\check{u}_{i}^+}_{\Omega_{i}^\pml} \leq \sum_{p=1}^{N-1} ( \tilde C \mu_Nkl/N)^{-(p+1)} \cdot \norme{\bar{u}_{i-1}^+-\check{u}_{i-1}^+}_{\Omega_{i-1}^\mathrm{PML}} \\
		 & \qquad+ C_{k,l,N} (kl/N)^2 \sum_{p=1}^{N-1} ( \tilde C \mu_Nkl/N)^{-(p+1)} \cdot e^{-\frac{1}{2}k\bar{\bar\si}} \norm{f}_{L^2(B_l)} \notag      \\
		 & \leq ( \tilde C \mu_Nkl/N)^{-(N+1)} \cdot \norme{\bar{u}_{i-1}^+-\check{u}_{i-1}^+}_{\Omega_{i-1}^\mathrm{PML}} \notag         \\
		 & \qquad+ C_{k,l,N} (kl/N)^2 ( \tilde C \mu_Nkl/N)^{-(N+1)} \cdot e^{-\frac{1}{2}k\bar{\bar\si}} \norm{f}_{L^2(B_l)}. \notag
	\end{align}
	Then the proof of the lemma is completed by induction.
\end{proof}

The following theorem is a direct consequence of Theorem~\ref{th3}, Lemma~\ref{lem7} and the fact that $u_i^\pm=\bar{u}_i^\pm$ in $\Omega_i\cup\Omega_{i+1}$.
\begin{theorem} \label{th_final}
	Let $\sigma_0>1$ be sufficiently large. Assume that $d\eqsim l/N$ and $kl/N\gtrsim 1$. Let $\check{u}_0^+=\check{u}_{N+1}^-=0$ in $B_L$ and $\check{u}(x)=-(\check{u}_{i-1}^++\check{u}_{i+1}^-)$ in $\Omega_i\cap B_L$ for all $i=1,\cdots,N$. Then we have
	\begin{align}\label{eqend1}
		\norme{\check{u}-\til{u}}_{B_L} & \lesssim \tilde C_{k,l,N} e^{-\frac{1}{2}k\bar{\bar\si}} \norm{f}_{L^2(B_l)},
	\end{align}
	where $\tilde C_{k,l,N}= (kl/N)^\frac72 ( \tilde C \mu_Nkl/N)^{-(N^2+2N-2)}$ and $\tilde C$ is independent of $f$, $k$, and $N$.
\end{theorem}

{\it Remark 3.1.} Let us take a close look at the estimate \eqref{eqend1} when $\al=0$ in Lemmas~\ref{lem6e} and \ref{Linfsup} (see Remark 2.1 (i)). The constant ${\tilde C}_{k,l,N}$ may be bounded as follows:
\eqn{ \tilde C_{k,l,N} \lesssim (kl/N)^\frac72 C(\si_0)^{N^2+2N-2} }
for some constant $C(\si_0)>1$, that is independent of $k,l,$ and $N$ but may polynomially depend on $\sigma_0$. Noting that $\bar{\bar\si}\eqsim\si_0d\eqsim\si_0l/N$, \eqref{eqend1} implies that
\eq{\label{eq_cklna0}
\norme{\check{u}-\til{u}}_{B_L} & \lesssim (c_2\si_0)^{c_3N^2}e^{-c_4kl/N\si_0} e^{-\frac{1}{4}k\bar{\bar\si}}\norm{f}_{L^2(B_l)}\notag\\
& \lesssim \bigg(\frac{c_2c_3N^2}{ec_4kl/N}\bigg)^{c_3N^2} e^{-\frac{1}{4}k\bar{\bar\si}}\norm{f}_{L^2(B_l)}
}
where $c_2, c_3$, $c_4>0$, and the invisible constant are independent of $k, l, N$, and $\sigma_0$. \eqref{eq_cklna0} implies that the PSTDDMb solution $\check{u}$ is a good approximation of the PML solution $\tilde u$ in $B_L$ (or the original scattering solution $u$ in $B_l$) if $k\si_0d\gg N^2\ln\big(\frac{c_2c_3N^2}{ec_4kl/N}\big)$.

\section{Numerical examples} \label{sec5} Noting from Theorems~\ref{th4} and \ref{th_final} (see Remark 3.1) that Algorithms~\ref{alg2} and \ref{alg5} produce good approximations to the original truncated PML solution in \eqref{eq_tpml} when $k\si_0d$ is sufficiently large, the PSTDDM and PSTDDMb can be used either as direct solvers or as preconditioners in the preconditioned GMRES method for solving the original truncated PML problem \eqref{eq_tpml}. In this section we present numerical experiments to verify the behavior of our PSTDDM as both direct solvers and preconditioners. The computations are all carried out in MATLAB with Intel(R) Xeon(R) CPU 2.5GH and 128GB memory.

We simply define the medium property by $\sigma_j(t)=\frac{\sigma_0}{d_j^2}(t-l_j)^2$ for $t\ge l_j$ and take $\sigma_0$ such that the exponential decaying factor $e^{-\frac12k\int_l^{l+d}\sigma_j(t)dt} \leq 10^{-3}$ in the following examples for simplicity.

The functions $\be_{i}^\pm(x_1),\ x_1\in\Omega_i,\ i=2,\cdots,N-1$, used in the source transfer algorithm are defined as
\begin{equation*}
	{\be_i^+}(x_1)=\begin{cases}
		1,  & \ze_i\leq x_1<\ze_i+\ze_i+\De\ze/4, \\
		\eta_i(x_1), & \ze_i+\De\ze/4\leq x_1< \ze_i+3\De\ze/4, \\
		0,  & \ze_i+3\De\ze/4\leq x_1\leq \ze_{i+1},
	\end{cases}
\end{equation*}
and ${\be_i^-}=1-{\be_i^+}$, where
\begin{align*}
	\eta_i(x_1)=1+\left(\frac{x_1-(\ze_i+\De\ze/4)}{\De\ze/2}\right)^4-2\left(\frac{x_1-(\ze_i+\De\ze/4)}{\De\ze/2}\right)^2.
\end{align*}
Clearly, $\be_{i}^{\pm}(x_1),\ i=2,\cdots,N-1$, are in $C^1(\Omega_i)$ and this fact avoids the discontinuity of ${\be_{i}^{\pm}}(x_1)'$ which may make $\hat{f}_{i}^\pm$ oscillate heavily.

Algorithms~\ref{alg2} and \ref{alg5} are discretized by using the finite element methods (FEM) as follows. We solve the truncated PML problems in Algorithms~\ref{alg2} and \ref{alg5} by the bilinear FEM on the Cartesian mesh consisting of small squares with side length $h$. The discrete source transfer operators are computed by replacing the truncated PML solutions with their corresponding FE approximations. We denote by $u_h$ the numerical solution obtained by the FE discretization of our PSTDDM Algorithm~\ref{alg2} or PSTDDMb Algorithm~\ref{alg5}. The numerical solution obtained by the preconditioned GMRES method using the discrete PSTDDM or PSTDDMb as a preconditioner is also denoted by $u_h$.

Let $u_f$ be the bilinear FEM solution to the original truncated PML problem~\eqref{eq_tpml} and let $u_I$ the bilinear interpolation of $u$ on the Cartesian mesh. We denote the relative errors of these numerical approximations in $H^1$-seminorms by
\begin{align*}
	e_h = & \frac{|u-u_h|_{H^1(B_l)}}{|u|_{H^1(B_l)}},\quad e_f = \frac{|u-u_f|_{H^1(B_l)}}{|u|_{H^1(B_l)}},\quad e_I = \frac{|u-u_I|_{H^1(B_l)}}{|u|_{H^1(B_l)}}.
\end{align*}
In order to compare the discrete PSTDDM(b) solution $u_h$ and the FE solution $u_f$, we denote the ratio of the error between $u_h$ and $u_f$ to the error of the FE solution by
\begin{align*}
	e_{hf} = \frac{|u_h-u_f|_{H^1(B_l)}}{|u-u_f|_{H^1(B_l)}}.
\end{align*}

\textbf{Example 4.1.} We solve the problem~\eqref{eq_tpml} for constant wave number with $f$ is given so that the exact solution is
\begin{equation*}
	u=\left\{ \begin{aligned}
		 & -r^3(r^3+3r^2-12r+9) H_0^{(1)}(kr), & r\le1, \\
		 & -H_0^{(1)}(kr),   & r>1.
	\end{aligned} \right.
\end{equation*}
Clearly, $u\in C^2(\R^2)$ and $\supp f\subseteq \set{x: |x|\le 1}$. We set $l_1=l_2=l=2$.

We first test the PSTDDM (Algorithm~\ref{alg2}). We set $d_1=0.2$ and $d_2=0.4$. The left graph of Figure~\ref{fig_errors_STDDM_layers} plots the relative errors of the FE solutions, the interpolations, and the discrete PSTDDM solutions for $k=12\pi, 50\pi, 100\pi$ and $N=10$. It is shown that the accuracies of the FE solution $u_f$ and the discrete PSTDDM solution $u_h$ are almost the same in all cases. On the finest mesh consisting of 34591233 nodal points, the backslash solver of MATLAB encounter the out-of-memory error due to the large number of degrees of freedom. Note that, for each wave number $k$, the FEM error starts to decay at a mesh size smaller than that for the interpolation, and that the gap increases as $k$ increases. Such a phenomenon is the so-called pollution effect \cite{bips95,bs00,w,zw,dw,lw19}. We also compare the discrete PSTDDM solutions with FE solutions in the right graph, which shows that the discrete PSTDDM solutions are indeed good approximations to the corresponding FE solutions.

Next we test performance of the PSTDDM (Algorithm~\ref{alg2}) as a preconditioner in the preconditioned GMRES method for solving the linear system from the finite element discretization of the original truncated PML problem \eqref{eq_tpml}. We set the relative residue tolerance in the GMRES solver to be $10^{-8}$. Figure~\ref{fig_errors_STDDM_gmres_layers} plots the number of preconditioned GMRES iterations (right) and the relative errors (left) of the FE solutions, the interpolations, and the preconditioned GMRES solutions for $k=12\pi, 50\pi, 100\pi$ and $N=10$. It is shown that the number of preconditioned GMRES iterations is bounded uniformly with respect to the wave number and the mesh size (even decreases as $h$ decreases) and that preconditioned GMRES method produces good approximations to the corresponding FE solutions.

To investigate the behavior of the PSTDDM as the number of layers increases, we solve the problem for the fixed mesh size $h=2l/6000$ and $k=50\pi,100\pi$. The left graph of Figure~\ref{fig_errors_NoLayers} plots the relative errors of PSTDDM solutions solved by the PSTDDM as a direct solver versus $\pi N/(lk)$, the reciprocal of the number of wavelength $2\pi/k$ per layer of width $2l/N$. It's shown that the relative errors almost remain unchanged as $N$ increases when $\pi N/(lk)\leq 1$, which behaves better than the theoretical estimate in Theorem~\ref{th_final}. We also use the PSTDDM as a preconditioner to solve the problem, and plot the number of preconditioned GMRES iterations versus $\pi N/(lk)$ in the right graph of Figure~\ref{fig_errors_NoLayers}. Although the number of iterations grows as the number of layers increases, it is small ($\le 5$) even for $N$ so large that the width of the layers equals to the wavelength.

\begin{figure}[htbp]
	\centering
	\includegraphics[width=0.45\textwidth]{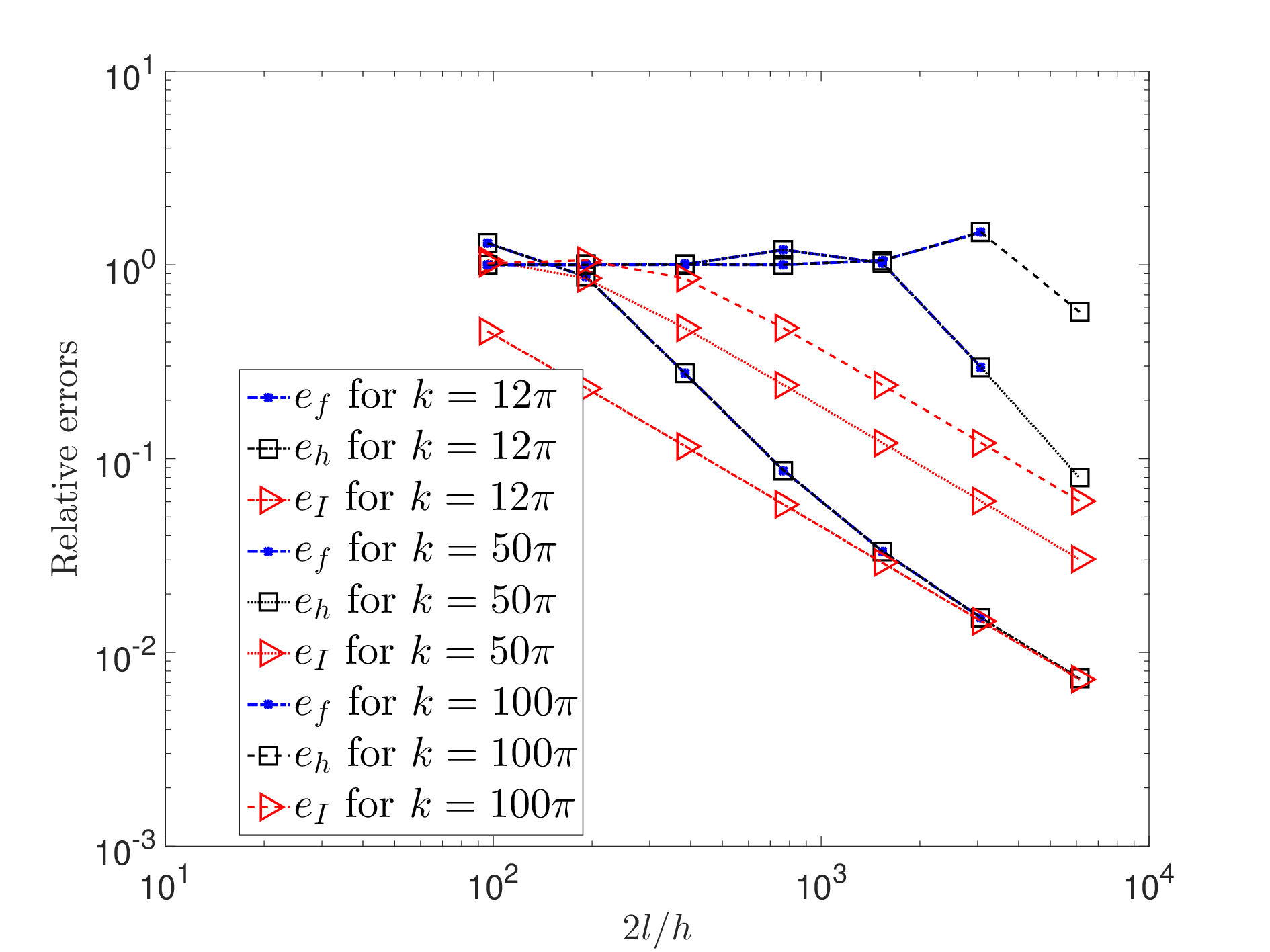}
	\includegraphics[width=0.45\textwidth]{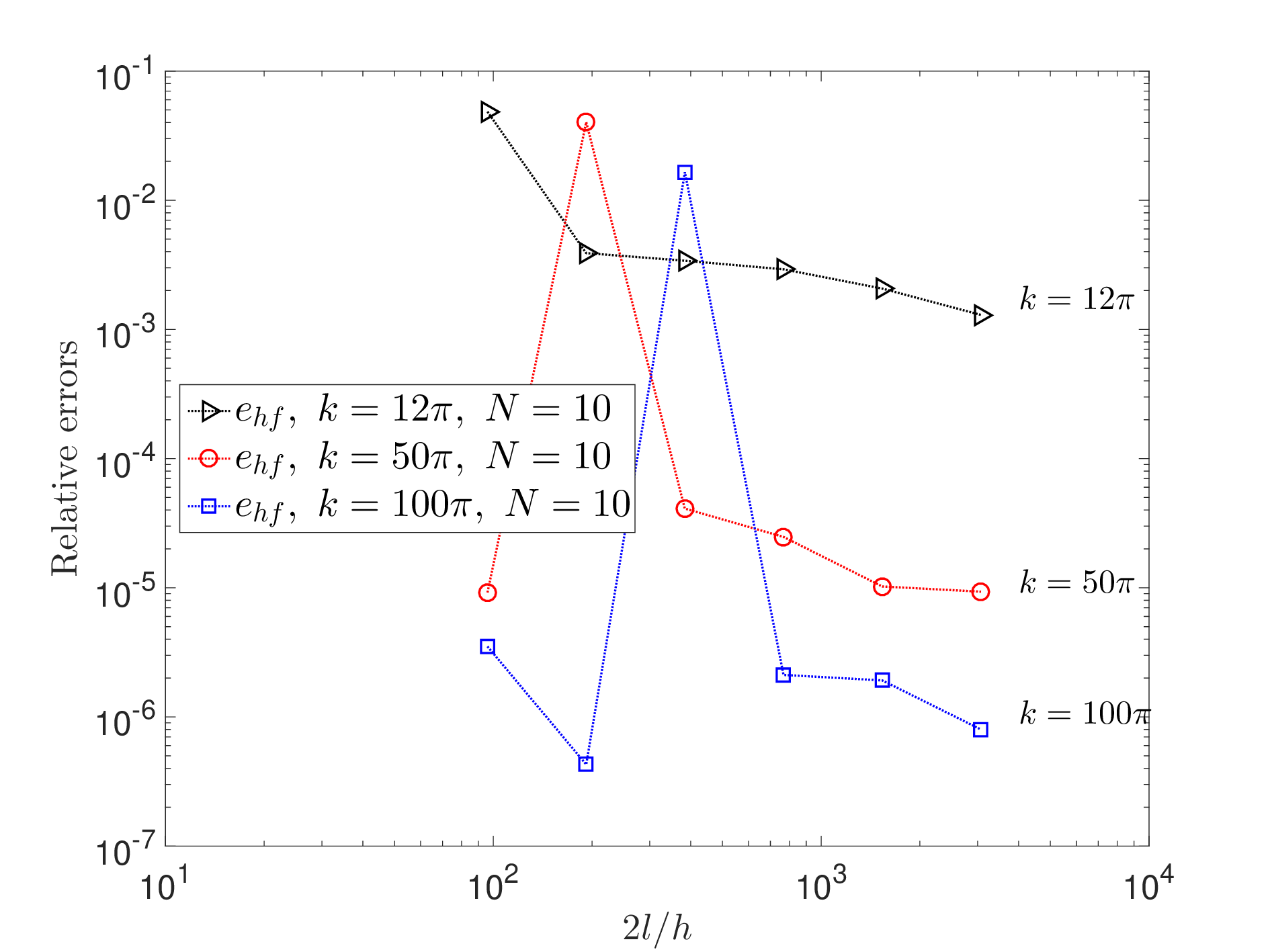}
	\caption{Example~1. Left: Log-log plot of the relative errors $e_f, e_h,$ and $e_I$ versus $2l/h$ for $N=10$ and $k=12\pi, 50\pi, 100\pi$. Right: Log-log plot of $e_{hf}$ versus $2l/h$.}
	\label{fig_errors_STDDM_layers}
\end{figure}

\begin{figure}[htbp]
	\centering
	\includegraphics[width=0.45\textwidth]{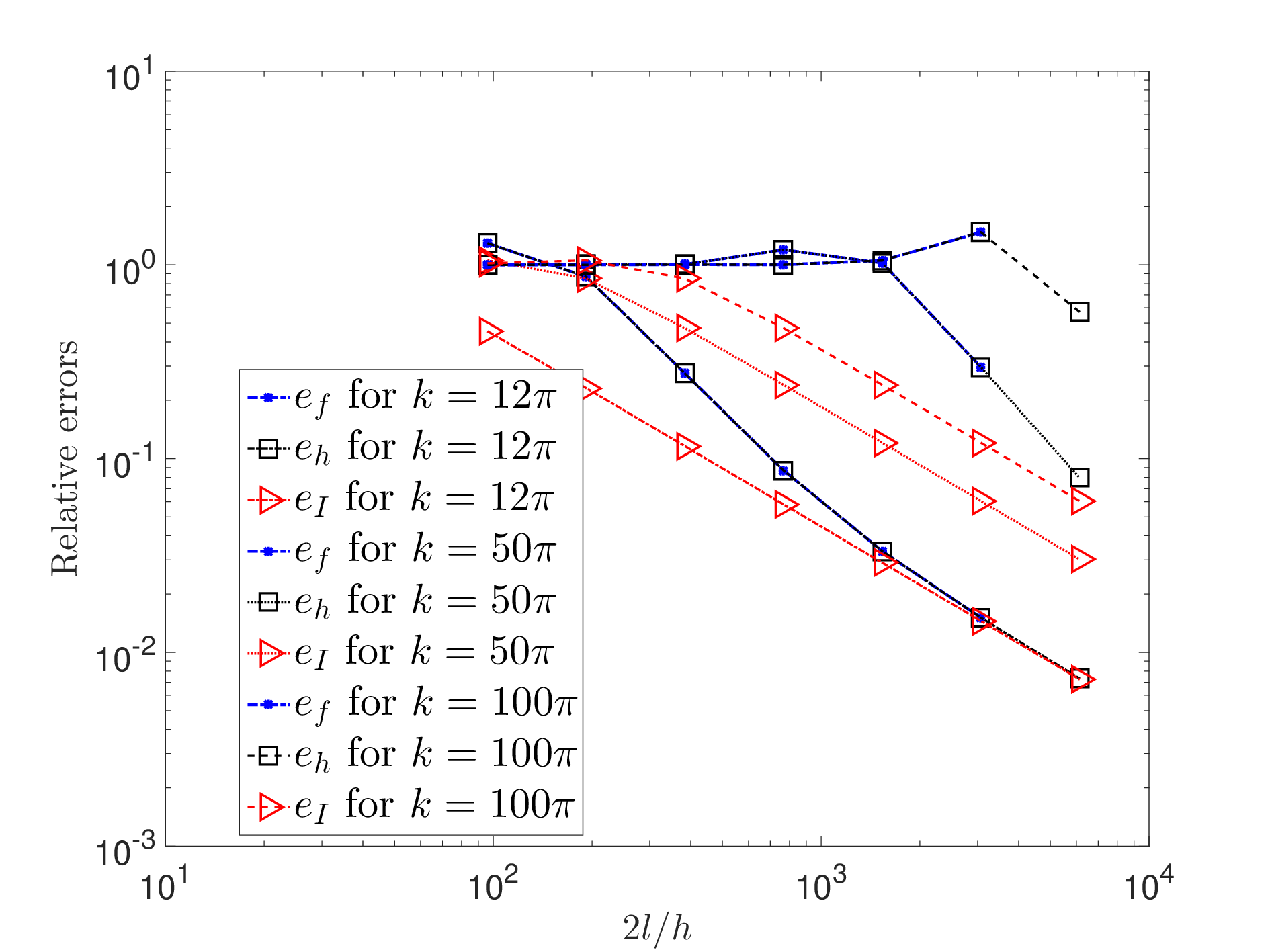}
	\includegraphics[width=0.45\textwidth]{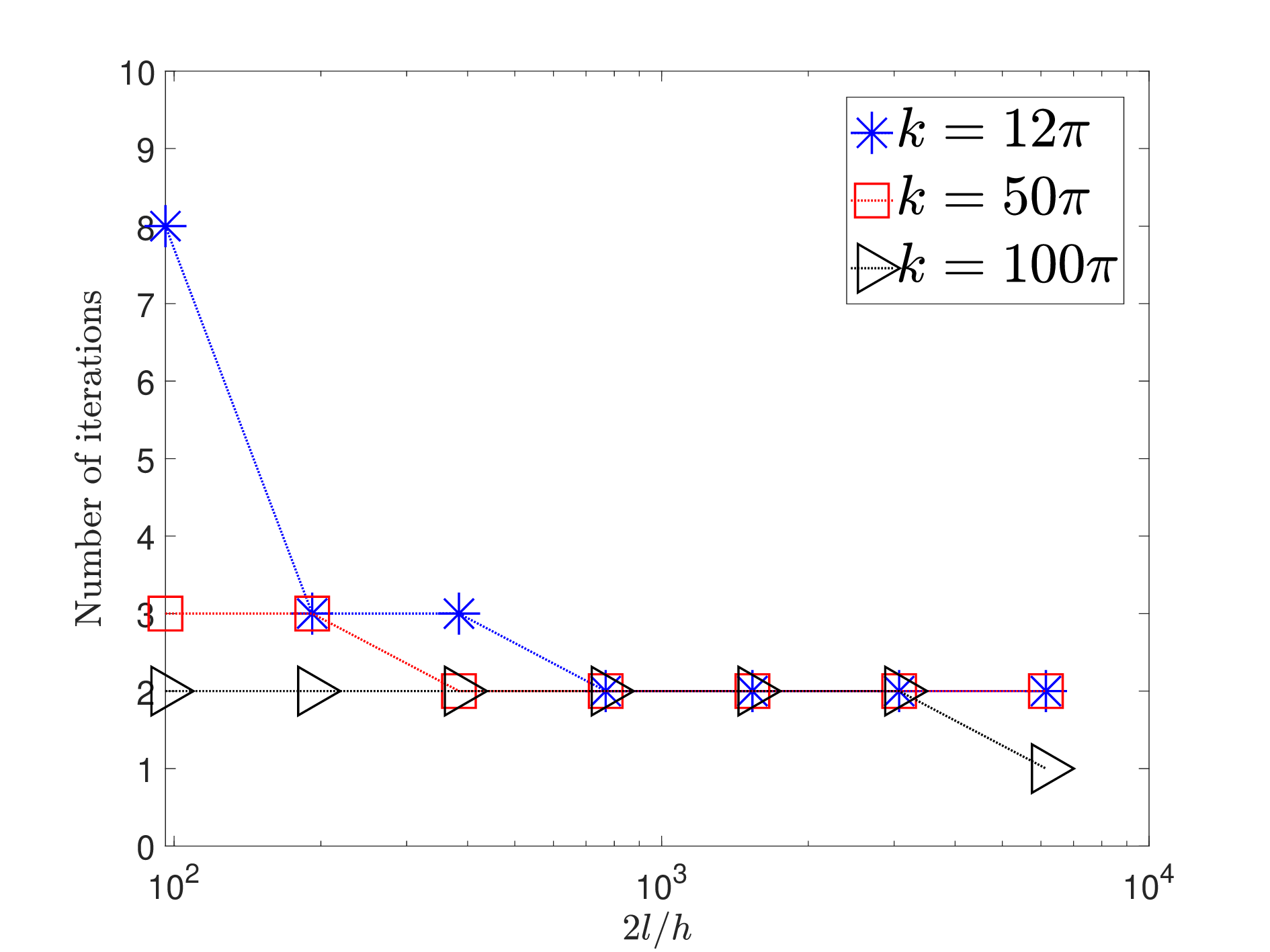}
	\caption{Example~1. Left: Log-log plot of the relative errors $e_f, e_h$(for the preconditioned GMRES solutions) and $e_I$ versus $2l/h$ for $N=10$ and $k=12\pi, 50\pi, 100\pi$. Right: the number of iterations of the preconditioned GMRES algorithm using PSTDDM as the preconditioner.}\label{fig_errors_STDDM_gmres_layers}
\end{figure}

\begin{figure}[htbp]
	\centering
	\includegraphics[width=0.45\textwidth]{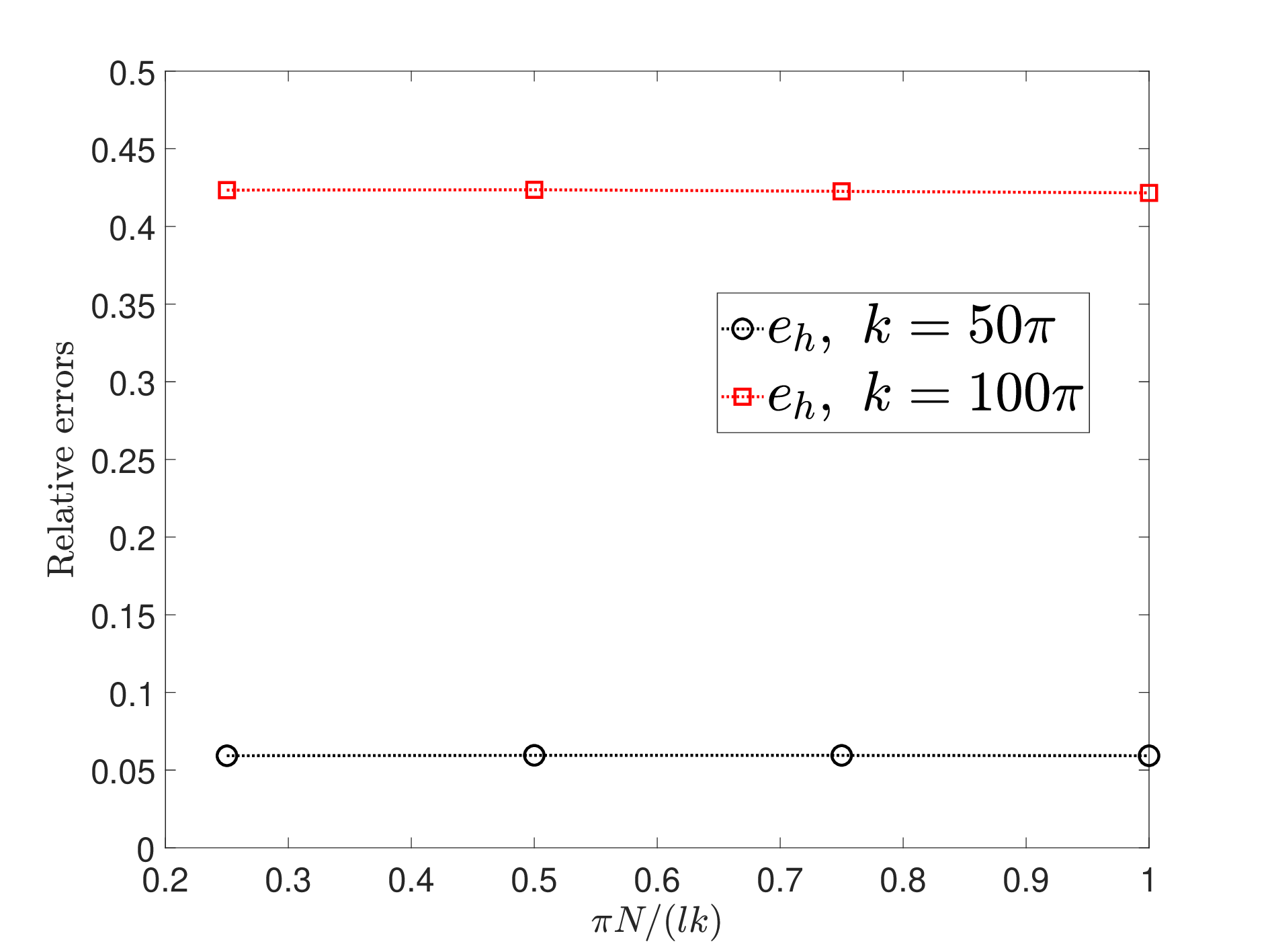}
	\includegraphics[width=0.45\textwidth]{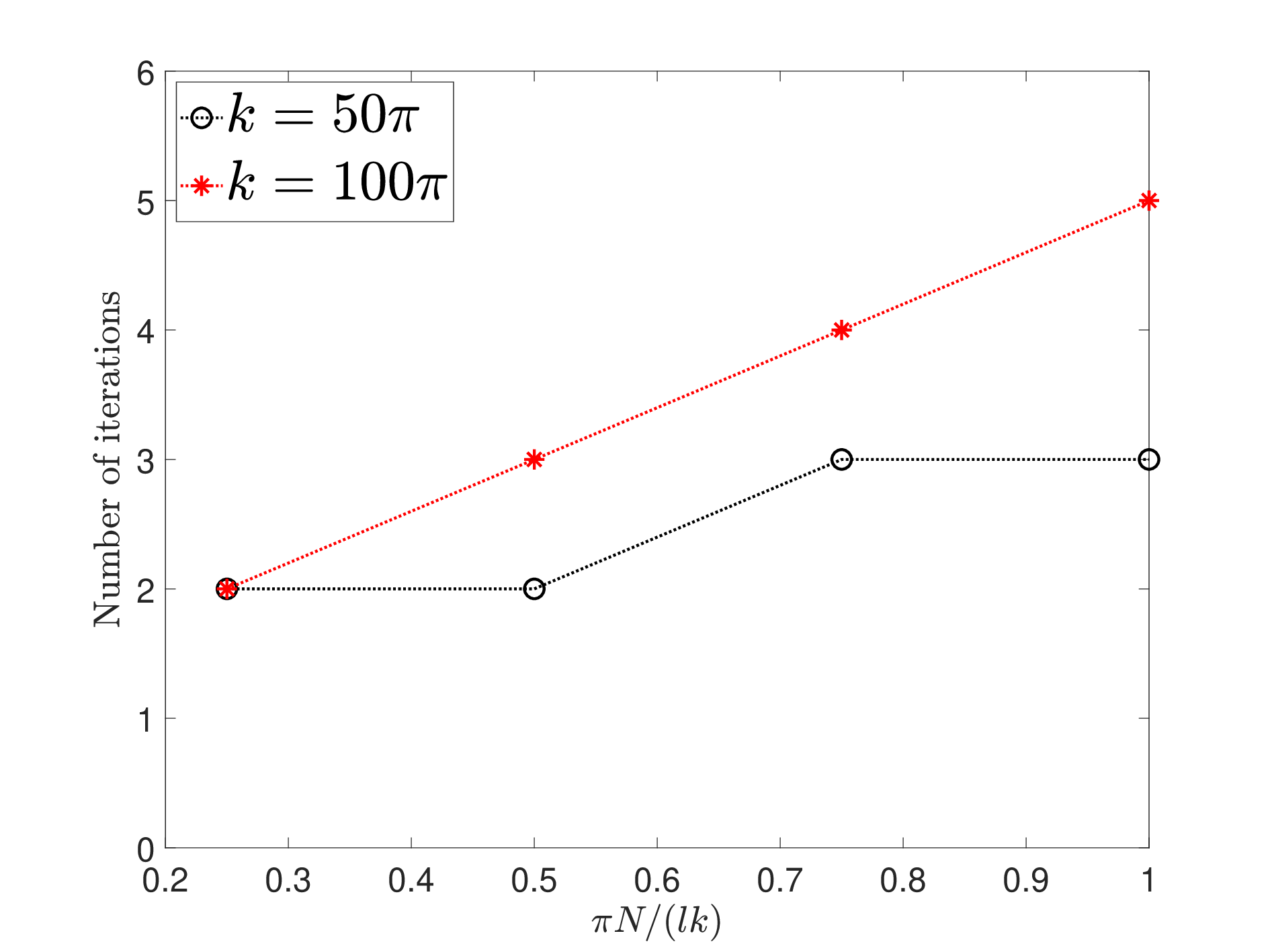}
	\caption{Example~1. Left: the relative errors $e_h$ of the solutions solved by the PSTDDM (see Algorithm~\ref{alg2}) as the direct solver for fixed $h=2l/6000$. Right: the number of iterations of the preconditioned GMRES algorithm using PSTDDM as the preconditioner for fixed $h=2l/6000$.}
	\label{fig_errors_NoLayers}
\end{figure}

Then we investigate the behavior of the PSTDDMb (see Algorithm~\ref{alg5}). We set $d_1=d_2=l/N$ and the other parameters about PML layers are still those provided at the beginning of this section.

In Figure~\ref{fig_errors_STDDM_squares}, we plot the relative errors of numrical solutions and compare the finite element solution and the PSTDDMb solution for $N=10$. We also use the PSTDDMb as the preconditioner in the preconditioned GMRES method where the relative residue tolerance is also $10^{-8}$. Figure~\ref{fig_errors_STDDM_gmres_suqares} plots the relative errors of the numerical solutions in the left graph and the number of iterations of the preconditioned GMRES for $k=12\pi,50\pi,100\pi$ in the right graph.

Finally, we plot in Figure~\ref{fig_errors_Nosquares} the relative errors of PSTDDMb solutions (left) and the number of iterations of the preconditioned GMRES method (right) versus $\pi N/(lk)$ for $k=50\pi,100\pi$ and fixed $h=2l/6000$. Again, the relative error of the PSTDDMb solution is independent of the number of layers while the number of iterations of the preconditioned GMRES method weakly depends on it.

\begin{figure}[htbp]
	\centering
	\includegraphics[width=0.45\textwidth]{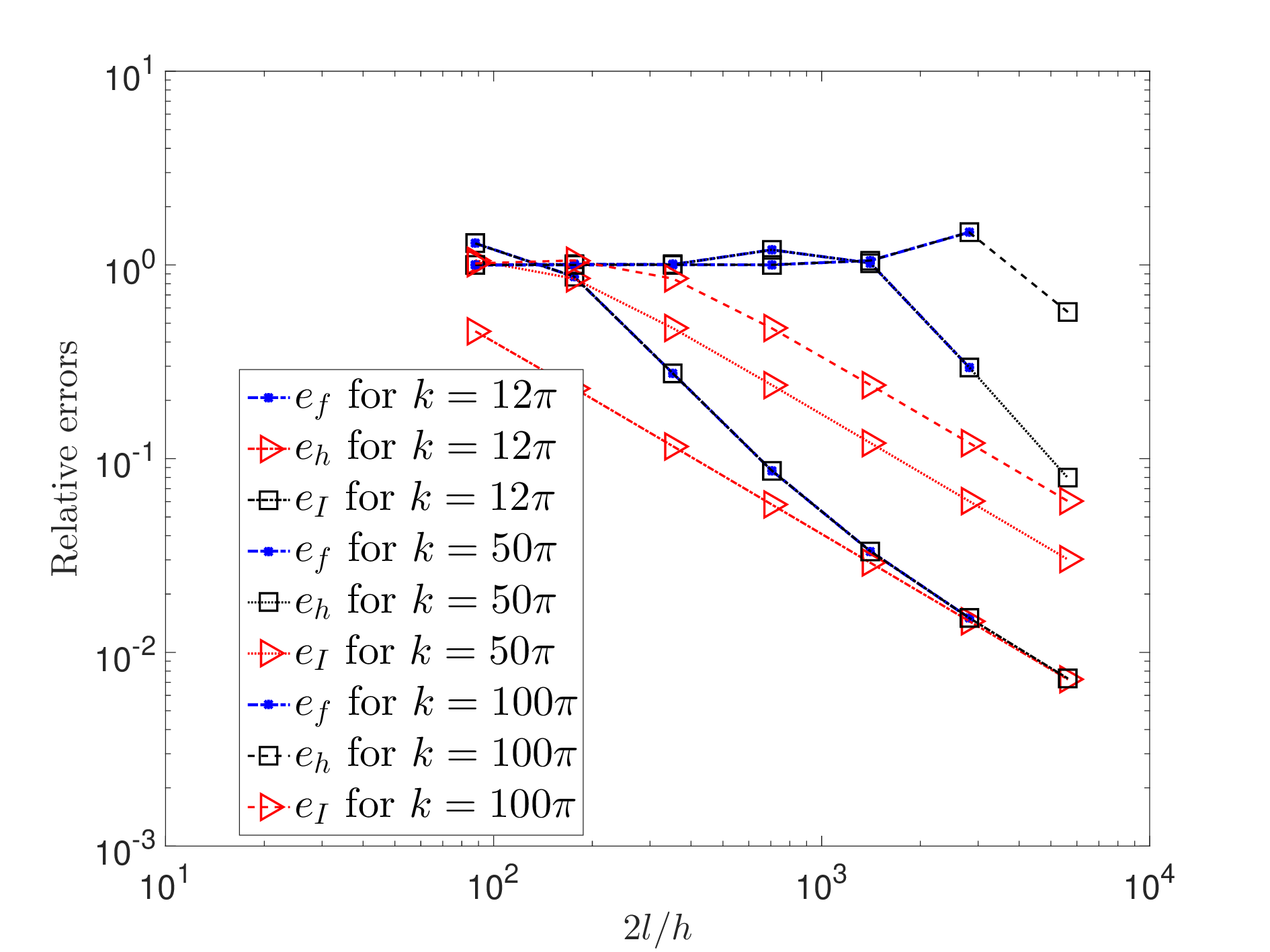}
	\includegraphics[width=0.45\textwidth]{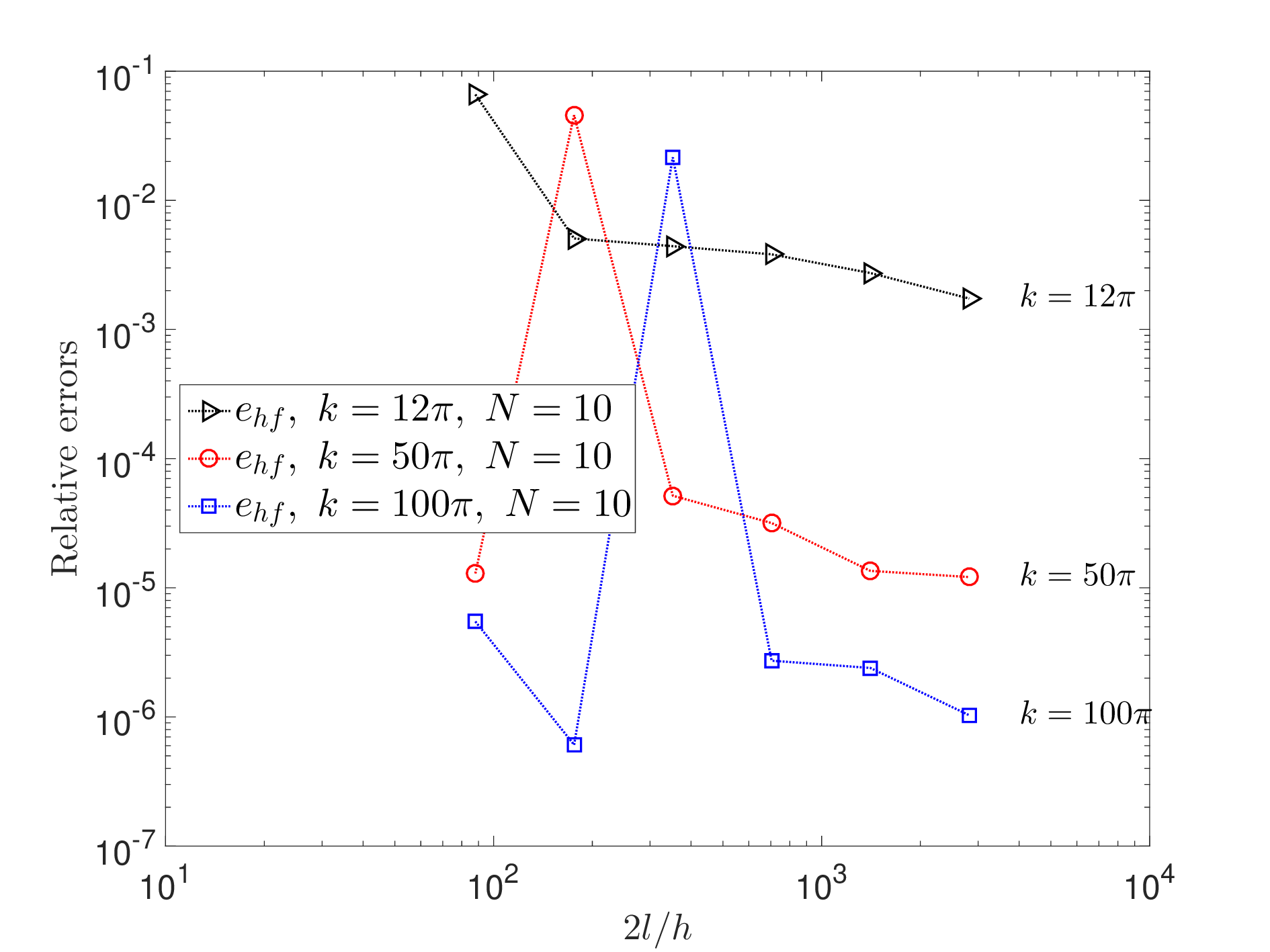}
	\caption{Example~1. The relative errors of the interpolations and the numerical solutions obtained by the PSTDDMb (see Algorithm~\ref{alg5}) as a direct solver with $N=10$.}
	\label{fig_errors_STDDM_squares}
\end{figure}

\begin{figure}[htbp]
	\centering
	\includegraphics[width=0.45\textwidth]{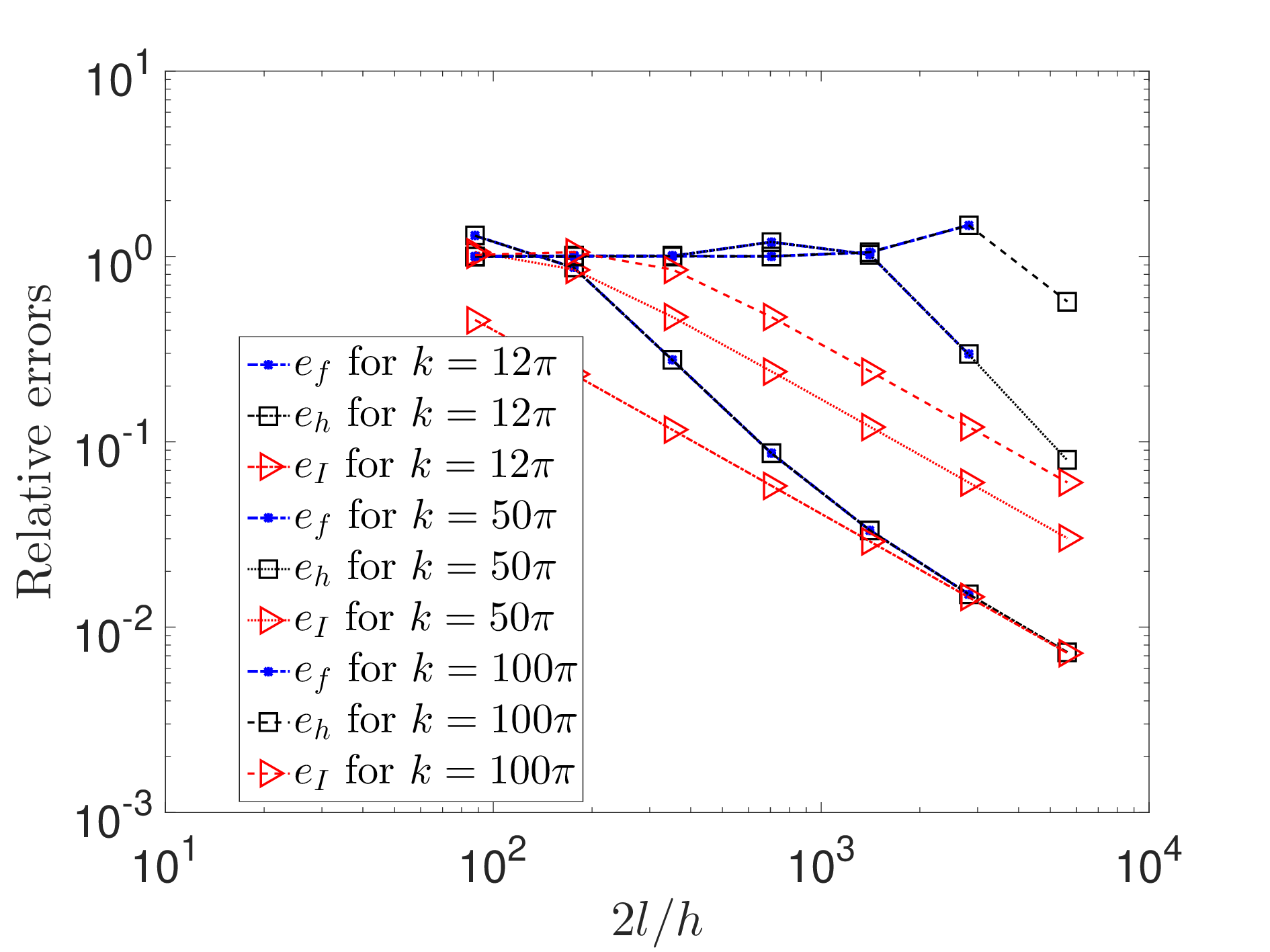}
	\includegraphics[width=0.45\textwidth]{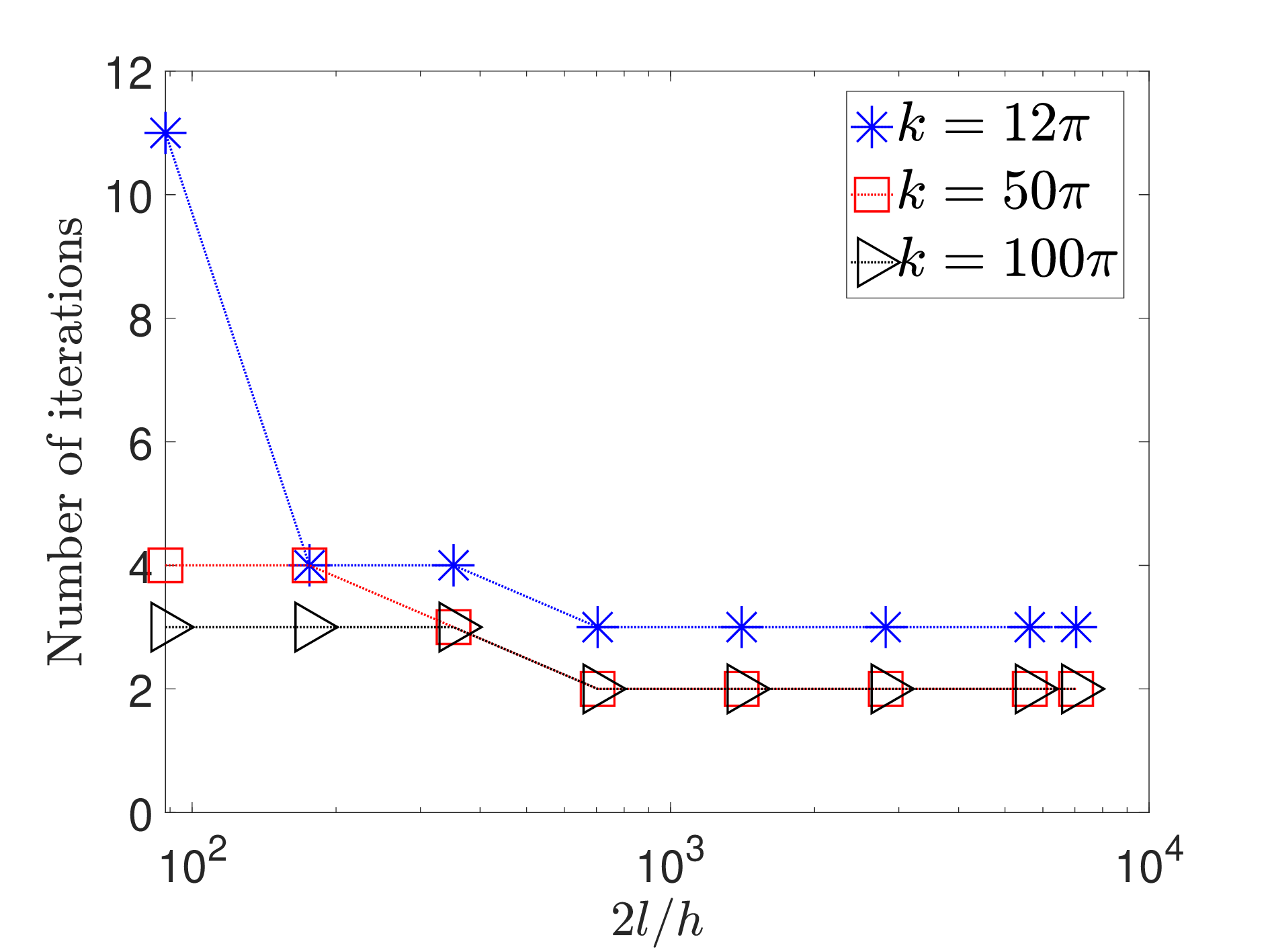}
	\caption{Example~1. The relative errors of the interpolations and the numerical solutions solved by the PSTDDMb (see Algorithm~\ref{alg5}) as a preconditioner with $N=10$ (left graph) and the number of iterations of the preconditioned GMRES algorithm using PSTDDMb as the preconditioner (right graph).}
	\label{fig_errors_STDDM_gmres_suqares}
\end{figure}

\begin{figure}[htbp]
	\centering
	\includegraphics[width=0.45\textwidth]{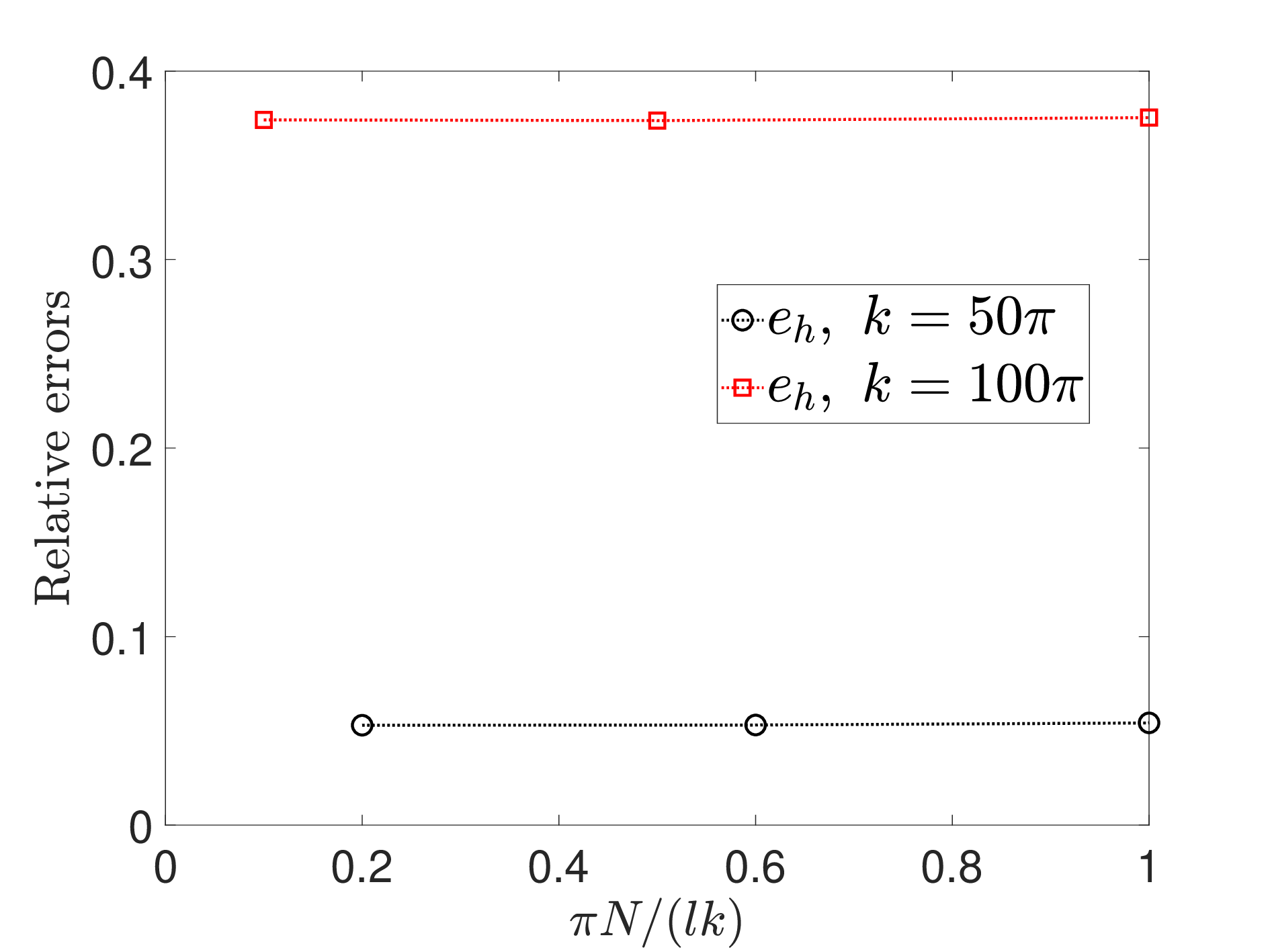}
	\includegraphics[width=0.45\textwidth]{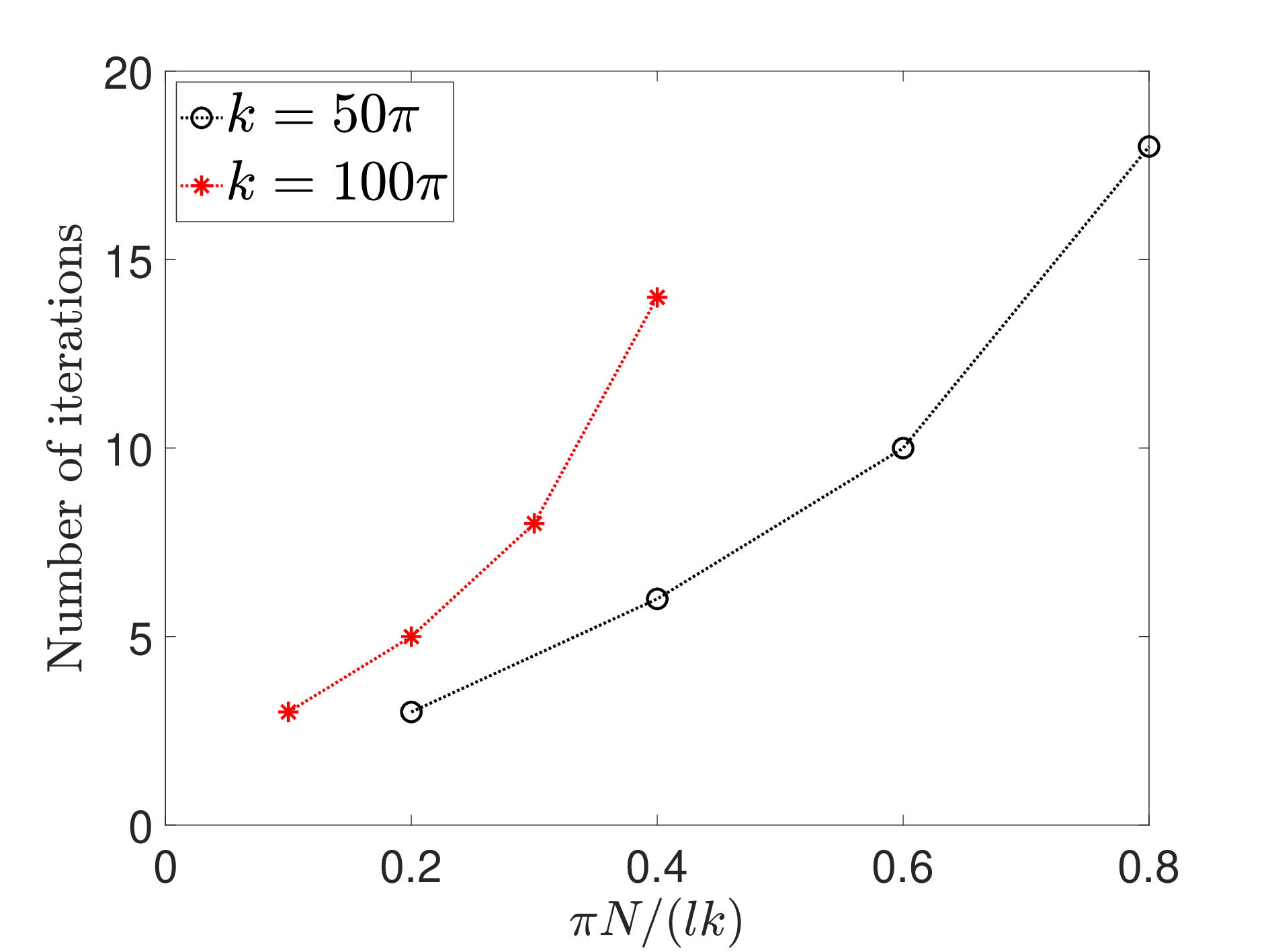}
	\caption{Example~1. Left: the relative errors $e_h$ of the PSTDDMb solutions (see Algorithm~\ref{alg5}) as the direct solver for $h=2l/6000$. Right: the number of iterations of the preconditioned GMRES algorithm using PSTDDMb as the preconditioner for $h=2l/6000$.}
	\label{fig_errors_Nosquares}
\end{figure}

	\textbf{Example 4.2.} We consider an example with heterogeneous wave number, which has been computed in \cite{cx} by STDDM. The wave number is $k(x)=\omega/c(x)$ where $\omega$ is the angular frequency and
	\begin{align*}
		c(x) = \frac43 - \frac23e^{-20(x_1^2+x_2^2)}
	\end{align*}
	is the velocity field. The source $f(x)$ is the narrow Gaussian point source
	\begin{align*}
f(x) = e^{-(\frac{4k}{\pi})^2((x_1-0.2)^2+(x_2-0.1)^2)}.
\end{align*}

	We solve the problem by using the PSTDDM (Algorithm~\ref{alg2}) and the PSTDDMb (Algorithm~\ref{alg5}) as the preconditioner in the GMRES, respectively. We choose $N=10$ and set $d_1=d_2=0.1$. The relative residue tolerance in the GMRES is set to be $10^{-10}$ as in \cite{cx}. Figure~\ref{fig_example3} shows the numbers of the preconditioned GMRES iterations for $k=100\pi$. The efficiencies of our PSTDDM and PSTDDMb as preconditioners in GMRES are quite similar to that of the STDDM proposed in \cite{cx}. We remark that, as a direct solver, the PSTDDM(b) does not give satisfactory results for such a problem with heterogeneous wave number, which is certainly not covered by our current theory.

\begin{figure}[htbp]
	\centering
	\includegraphics[width=0.45\textwidth]{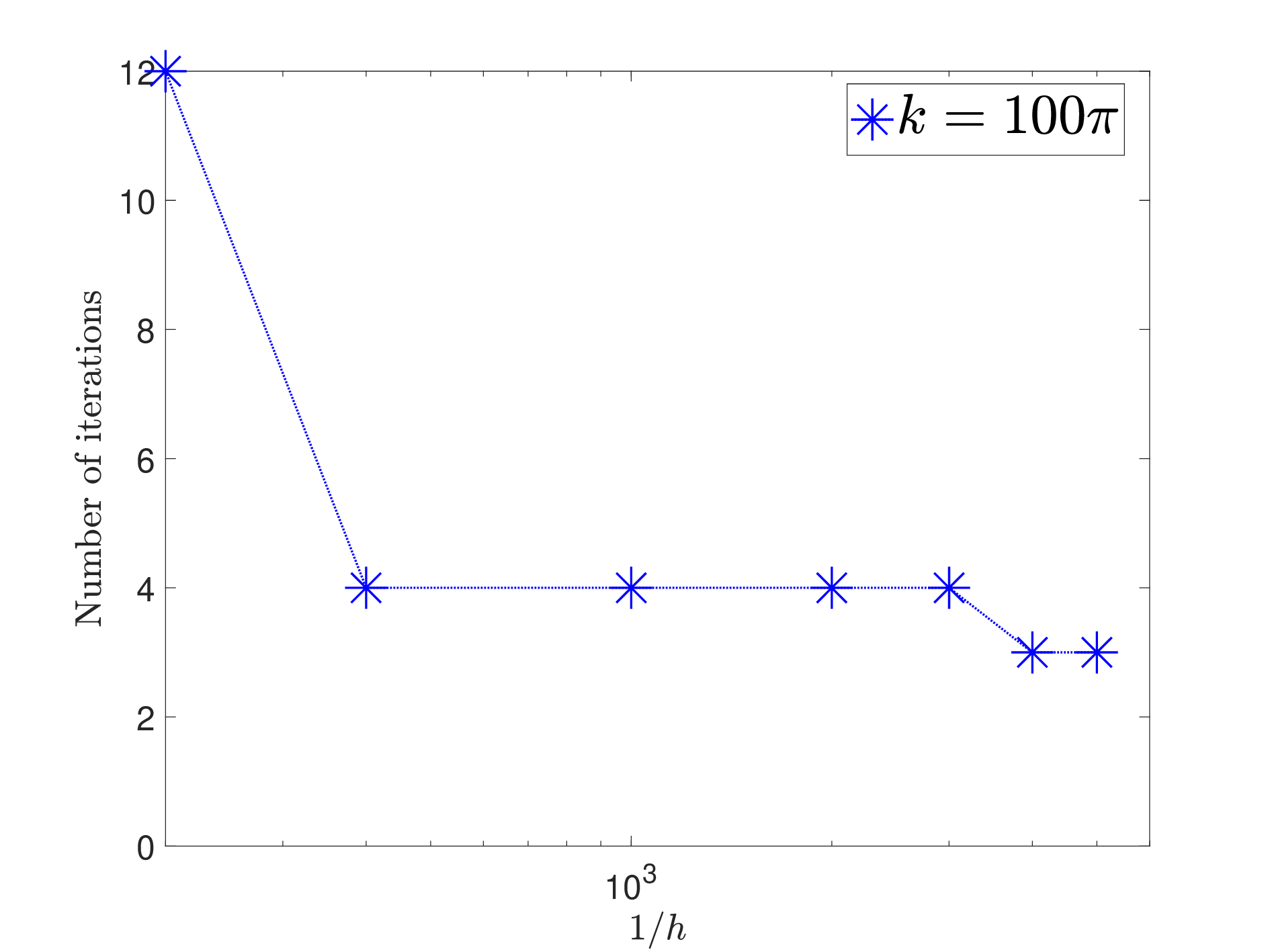}
	\includegraphics[width=0.45\textwidth]{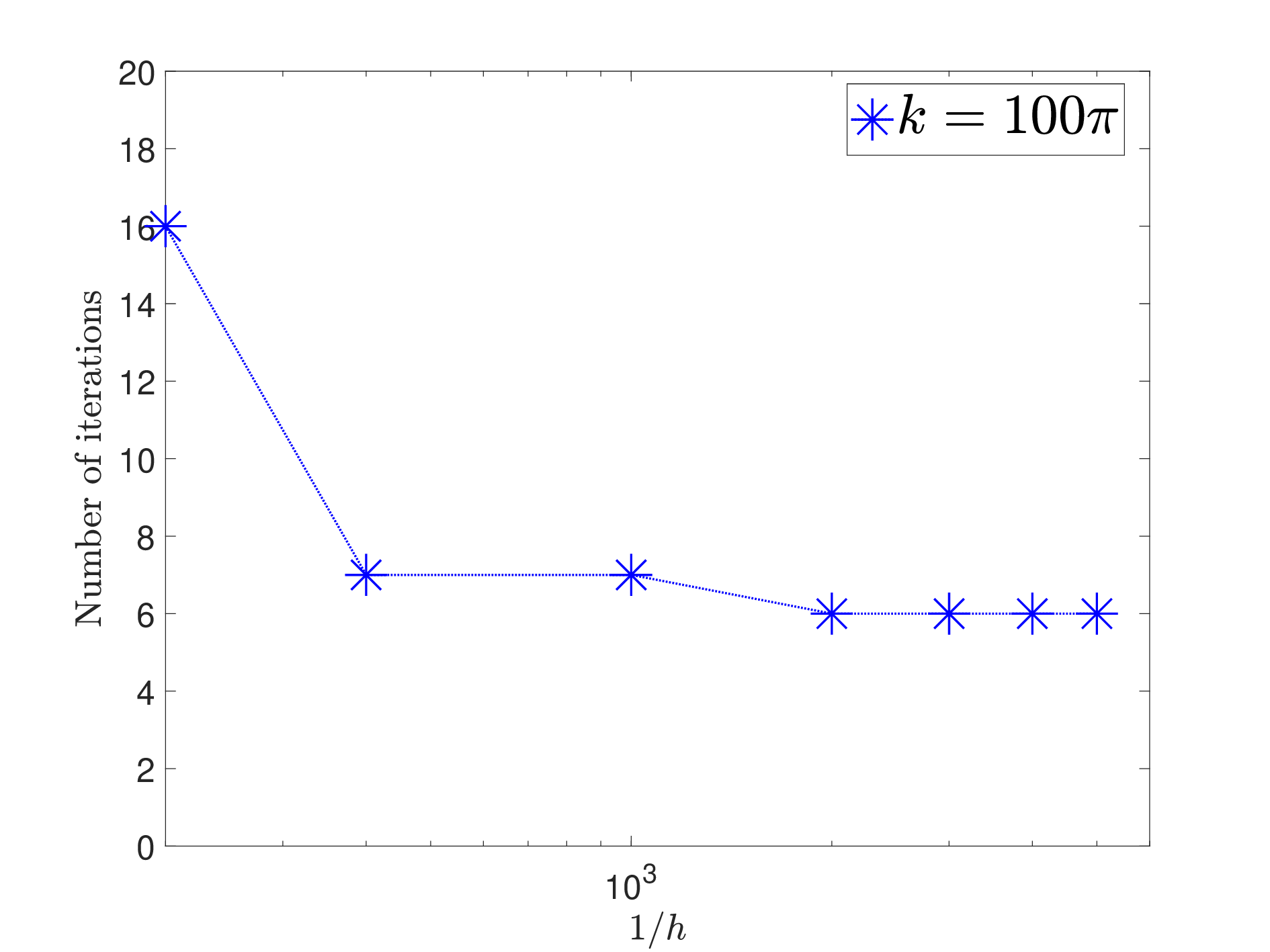}
	\caption{Example~3. Left : the number of iterations of the preconditioned GMRES algorithm using PSTDDM (Algorithm~\ref{alg2}) as the preconditioner for $N=10$. Right: the number of iterations of the preconditioned GMRES algorithm using PSTDDMb (Algorithm~\ref{alg5}) as the preconditioner for $N=10$.}
	\label{fig_example3}
\end{figure}

\bibliographystyle{siam}
\bibliography{ref_ipfemHH}

\end{document}